\documentclass[12pt]{amsart}
\usepackage[utf8]{inputenc}
\usepackage{amssymb}
\usepackage{amsmath}
\usepackage{amsthm}
\usepackage[T1]{fontenc}
\usepackage{mathtools}
\usepackage{wasysym}
\usepackage{verbatim}
\usepackage{graphicx}
\usepackage{color}
\usepackage[margin=1.0in]{geometry}
\usepackage{float}
\usepackage{hyperref}
\usepackage{enumitem,kantlipsum}

\theoremstyle{plain}
\newtheorem{theorem}{Theorem}[section]
\newtheorem{lemma}[theorem]{Lemma}

\newtheorem{question}[theorem]{Question}
\newtheorem{corollary}[theorem]{Corollary}

\newtheorem{fact}[theorem]{Fact}
\theoremstyle{definition}
\newtheorem{remark}[theorem]{Remark}

\newtheorem{definition}[theorem]{Definition}

\DeclareMathOperator{\DOWN}{DOWN}
\DeclareMathOperator{\UP}{UP}
\DeclareMathOperator{\lift}{lift}

\begin{document}

\newcommand{\mb}[1]{\mathbb{#1}}
\newcommand{\mc}[1]{\mathcal{#1}}
\newcommand{\Gr}{\text{Gr}}
\newcommand{\Mat}{\text{Mat}}
\newcommand{\Le}{\scalebox{-1}[1]{L}}
\newcommand{\vcm}[1][1]{\vspace*{#1 cm}}
\newcommand{\tcd}{2\leftrightarrow 2}
\newcommand{\mn}{\Delta_{\text{min}}}
\newcommand{\mx}{\Delta_{\text{max}}}
\newcommand\todo[1]{\textcolor{red}{TODO: #1}}

\title{Flip cycles in plabic graphs}

\author{Alexey~Balitskiy{$^\clubsuit$}}

\email{\vspace{-0.3cm}balitski@mit.edu}

\author{Julian~Wellman{$^\spadesuit$}}

\email{wellman@mit.edu}

\address{{$^\clubsuit$}{$^\spadesuit$} Dept. of Mathematics, Massachusetts Institute of Technology, 182 Memorial Dr., Cambridge, MA 02142, USA}
\address{{$^\clubsuit$} Institute for Information Transmission Problems RAS, Bolshoy Karetny per. 19, Moscow, Russia 127994}

\thanks{{$^\clubsuit$} Supported by the Russian Foundation for Basic Research grant 18-01-00036.}


\begin{abstract}
Planar bicolored (plabic) graphs are combinatorial objects introduced by Postnikov to give parameterizations of the positroid cells of the totally nonnegative Grassmannian $\Gr^{\geq 0}(n,k)$. Any two plabic graphs for the same positroid cell can be related by a sequence of certain moves. The \emph{flip graph} has plabic graphs as vertices and has edges connecting the plabic graphs which are related by a single move. A recent result of Galashin shows that plabic graphs can be seen as cross-sections of zonotopal tilings for the cyclic zonotope $Z(n,3)$. Taking this perspective, we show that the fundamental group of the flip graph is generated by cycles of length 4, 5, and 10, and use this result to prove a related conjecture of Dylan Thurston about triple crossing diagrams. We also apply our result to make progress on an instance of the generalized Baues problem.
\end{abstract}

\maketitle



\section{Introduction}\label{Intro}

A \emph{flip graph} for our purposes is the graph whose vertices form the set of all diagrams of a particular class, and whose edges correspond to \emph{flips} in these diagrams, which are mutations which transform one diagram into a similar diagram with one small thing changed. A common question to ask is if the flip graph is connected, that is, can any two objects in the set be related by a sequence of flips? For Postnikov's moves on plabic graphs~\cite{Pos06}, the answer is affirmative. In this paper we investigate the next natural topological question for some flips graphs that are known to be connected: is there a nice set of ``simple cycles'' which generate the fundamental group of the flip graph?

Possibly the most famous example of a flip graph is that of triangulations of an $n$-gon, whose flip graph forms the 1-skeleton of the \emph{Stasheff associahedron}. Although not novel (known at least since Stasheff's work~\cite{Sta63}), a nice corollary of our result is that the fundamental group of the 1-skeleton of the associahedron is generated by cycles of length four and five. Another famous flip graph has \emph{domino tilings} of a planar region as its vertices; in \cite{Thu90} it is proved that the flip graph is connected (provided that the region is simply connected) through a \emph{height function} on tilings. The methods used in this paper for zonotopal tilings are somewhat reminiscent of the height function idea.
Dylan Thurston~\cite{ThuD17} introduced triple crossing diagrams, which are a generalization of domino tilings, proved that the flip graph is connected, and made a conjecture about the fundamental group of the flip graph. One of the results of this paper is a proof of that conjecture. 

We will study several different flip graphs, whose objects and flips are as follows:\begin{itemize}
\item Fine zonotopal tilings of the cyclic zonotope $Z(n,d)$, with flips corresponding to switching between the two tilings of $Z(d+1,d)$.
\item Reduced trivalent plabic graphs (or plabic triangulations) for a given strand connectivity, with flips corresponding to the moves (M1)--(M3) in Figure~\ref{plabicmovespng}. 
\item Reduced bipartite plabic graphs (or plabic tilings) for a given connectivity, considered modulo the contraction/uncontraction moves (see Figure~\ref{plabicmovespng}). The flips are only given by the square move (M2).
\item Triple crossing diagrams for a given connectivity, with flips being $\tcd$ moves (see Figure~\ref{tcdflipng}).
\end{itemize}

For definitions of these objects and flips, see Section~\ref{Z} for zonotopal tilings, Section~\ref{plabic} for plabic graphs and tilings, and Section~\ref{tcd} for triple crossing diagrams. For those familiar with the language of weakly separated collections (which we avoid in this paper), it is worth mentioning how these are related to our objects.
\begin{itemize}
    \item Reduced plabic graphs (and plabic tilings) are in correspondence with weakly separated collections in $\binom{[n]}{k}$. The parameters $n$ and $k$ here are the size and the helicity of the corresponding strand permutation (as defined in Section~\ref{plabic}).
    \item Triple crossing diagrams are in correspondence with compatible pairs of a weakly separated collection in $\binom{[n]}{k}$ and a weakly separated collection in $\binom{[n]}{k+1}$. They are also in bijection with plabic graphs with white vertices of degree three (or plabic tilings with a specified white triangulation), or with plabic graphs with black vertices of degree three (plabic tilings with a specified black triangulation). The $\UP$ and $\DOWN$ maps (see Definition~\ref{updown}) switch the latter two descriptions.
    \item Reduced trivalent plabic graphs (or plabic triangulations) are in correspondence  with compatible triples of weakly separated collections in $\binom{[n]}{k-1}$, $\binom{[n]}{k}$, and $\binom{[n]}{k+1}$.
\end{itemize}

Our main theorem regards generating sets for the fundamental group of the flip graphs when considered as a 1-complex. Though we later phrase our theorems as proving that a 2-complex made out of the flip graph with certain 2-cells glued is simply connected, here we will simply state the sizes of the cycles which generate the fundamental group.

\begin{theorem}\label{mainthm}
The fundamental groups for our flips graphs on the following objects are generated by cycles with sizes as follows \begin{enumerate}
\item Fine zonotopal tilings of $Z(n,d)$, by cycles of sizes 4 and $2d+4$.
\item Reduced trivalent plabic graphs, by cycles of sizes 4, 5 (two types), and 10 (two types).
\item Reduced bipartite plabic graphs, by cycles of sizes 4 and 5.
\item Triple crossing diagrams, by cycles of sizes 4, 5, and 10.
\end{enumerate}
\end{theorem}

The first part is the subject of Section~\ref{Z}, and is a result of using Ziegler's results on the higher Bruhat order poset \cite{Zie91} to generalize the proof for $d=2$ given by Henriques and Speyer~\cite{HenSpe10} (for the flips in the rhombus tilings of the regular $2n$-gon). The second result is completely new to our knowledge, and uses Galashin's~\cite{Gal17} interpretation of plabic graphs as cross-sections of fine zonotopal tilings. An alternate proof for the case where the plabic graphs have a cyclic strand connectivity (corresponding to the totally positive Grassmannian $\Gr^{>0}(n,k)$, \cite{Pos06}) is given in the appendix, and uses the first result. The relationship between zonotopal tilings and plabic graphs is elaborated on further in Section~\ref{plabic}, and the result is proven in Section~\ref{cycles}. The last two results are corollaries of the second, with the fourth result proving a conjecture of Dylan Thurston~\cite[Conjecture 21]{ThuD17}, see Section~\ref{tcd}. 

In the case $d=3$, the fine zonotopal tiling flip graph is generated by squares and decagons. The vertices of the decagons correspond to fine zonotopal tilings of $Z(5,3)$, one of which is shown in Figure~\ref{zonotopaltilipng}. The plabic graph cycles of length 5 and 10 occur in the cross-sections of the tilings as shown in Figure~\ref{plabicyclepng}.

\begin{figure}\centering
\includegraphics[width=0.8\textwidth]{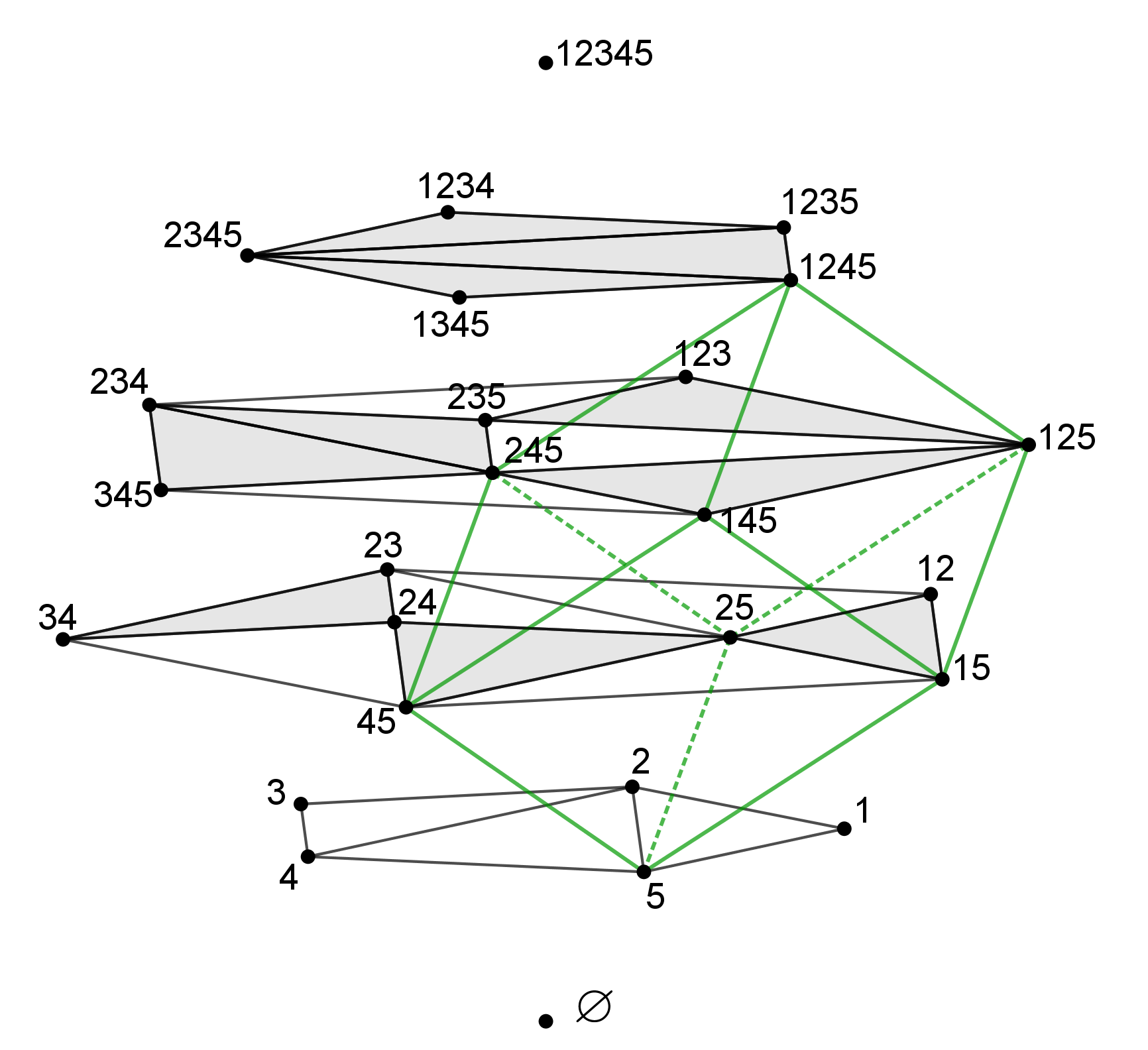}
\caption{A fine zonotopal tiling of $Z(5,3)$ with the tile $\tau_{(\{5\},\{3\})}$ highlighted in green. The four nontrivial cross-sections shown (namely, their bottom view) are planar duals to plabic graphs, and exhibit four of the five types of 2-cells in $\mb X_{\pi(n,k)}$ (see Theorem~\ref{plabicycles}). Redrawn from \cite[Figure~2]{Gal17}}
\label{zonotopaltilipng}
\end{figure}

The questions we discuss regard the relations between flips in our flip graphs. One can further ask about relations between relations, and so on. A natural way to pose this question formally involves the topology of the poset, whose minimal elements are our configurations, rank 1 elements are the flips, rank 2 elements are relations between flips, etc. A formal way to do this is known as the generalized Baues problem, which we discuss in Section~\ref{baues}. We prove there that the Baues poset of the Grassmannian graphs, conjectured~\cite{Pos18} to have the homotopy type of sphere,\footnote{While this paper was under review, this conjecture was confirmed~\cite{olarte2019} by Olarte and Santos.} is simply connected.

\section{Cycles for Zonotopal Tilings}\label{Z}

\begin{definition}\label{cyczonotope}
Let $v_1,\ldots,v_n \in \mb R^d$ be any collection of distinct vectors on the dimension $d$ \emph{moment curve} parameterized by $(1,t,t^2,\ldots,t^{d-1})$ (the indexing of the $v_i$ is increasing in $t$).
The \emph{cyclic zonotope} $Z(n,d)$ consists of all points which can be written as $\sum\limits_{i=1}^n c_iv_i$ for some $(c_i)_{i=1}^n \in [0,1]^n$, that is, $Z(n,d)$ is the Minkowski sum of the intervals $[0,v_i]$.
\end{definition} 

\begin{remark}
Sometimes cyclic zonotopes (and cyclic polytopes) are defined more generally, cf.~\cite[Definition~9.1]{Pos18}. Everything we prove could be related to those slightly more general objects as well.
\end{remark}

Following Galashin~\cite{Gal17}, we define tilings of the cyclic zonotope as collections of signed subsets. A pair $X = (X^+,X^-)$ of disjoint subsets of $[n]$ is called a signed subset of $[n]$, and we also define $X^0 := [n] \setminus (X^+ \sqcup X^-)$. Then the signed subsets are exactly the strings in $\{+,-,0\}^n$. For $X$ a signed subset, the tile $\tau_X$ consists of all points which can be written as $\sum\limits_{i\in X^+} v_i + \sum\limits_{j\in X^0} c_jv_j$ for some $(c_j)_{j\in X^0} \in [0,1]^{|X^0|}$. 
\begin{definition}\label{tiling}
A collection $\Delta$ of signed subsets of $[n]$ is called a \emph{fine zonotopal tiling} of $Z(n,d)$ provided that\begin{enumerate}
\item $Z(n,d)= \bigcup\limits_{X\in \Delta} \tau_X$,
\item Whenever $\tau_X \cap \tau_Y \neq \emptyset$ for $X,Y\in \Delta$, there exists $Z \in \Delta$ such that $\tau_X \cap \tau_Y = \tau_Z$ is a face of both $\tau_X$ and $\tau_Y$, and
\item For all $X \in \Delta$, we have $|X^0| \leq d$.
\end{enumerate}
\end{definition}

When the third condition fails, $\Delta$ is a zonotopal tiling but is not \emph{fine}. For any fine zonotopal tiling $\Delta$ and set $S \in \binom{[n]}{d}$ there a unique $X\in \Delta$ with $X^0 = S$ (cf.~\cite[item~(56)]{Shephard1974}). These tiles $\tau_X$ with $|X^0| = d$ are $d$-dimensional parallelotopes, and completely determine the tiling; the tiles $\tau_X$ with smaller $X_0$ sets are the lower-dimensional faces of the parallelotopes. Fine zonotopal tilings $\Delta$ of $Z(n,d)$ can be related to each other through a series of mutations. Geometrically, these mutations consist of finding a finely-tiled copy of $Z(d+1,d)$ inside $\Delta$, which has only two fine tilings, and flipping the way it is tiled. We will use the combinatorial definition in terms of signed subsets. Suppose that $S \in \binom{[n]}{d+1}$ has elements $i_1 < \cdots < i_{d+1}$. Then there exists a unique sequence of signed subsets $X_1,\ldots,X_{d+1} \in \Delta$ with $X_\ell^0 = S \setminus \{i_\ell\}$. Let $s_\ell = 1$ if $i_\ell \in X_\ell^+$, and $s_\ell = 0$ otherwise (when $i_\ell \in X_\ell^-$). Finally, let $S_\ell^+ = X_\ell^+ \setminus \{i_\ell\}$. We will use these definitions to check whether there is a tiled copy of $Z(S,d)$ inside $\Delta$ to flip.


\begin{definition}\label{zonotopeflip}
A \emph{flip} at the set $S$ is available in a fine zonotopal tiling $\Delta$ if $S_i^+ = S_j^+$ for all $i,j \in [d+1]$ and $s_{\ell} \neq s_{\ell+1}$ for all $\ell \in [d]$. The result of a flip at the set $S$ is a new fine zonotopal tiling $\Delta'$, which can be formed from $\Delta$ by swapping each element $i_\ell \in S$ between $X_\ell^+$ and $X_\ell^-$, changing the values of all the $s_\ell$. The other $d$-dimensional parallelotopes are unchanged, and the smaller tiles are again the faces of these parallelotopes.
\end{definition}
The \emph{flip graph} is the graph which has fine zonotopal tilings as vertices and edges connecting those tilings which are related by a single flip. It is a fact that any two fine zonotopal tilings of $Z(n,d)$ can be related by a series of flips, so the flip graph is connected, as we will see.

Henriques and Speyer (\cite[Proposition 3.14]{HenSpe10}) prove that the fundamental group of the flip graph of $Z(n,2)$ as a 1-complex is generated by 4-cycles and 8-cycles, where the 4-cycles correspond to pairs of commuting flips and the 8-cycles correspond to copies of $Z(4,2)$. In this section we generalize this result to any dimension using Ziegler's~\cite{Zie91} results on the \emph{higher Bruhat order}. Ziegler~\cite{Zie91} shows (with different language) that the flip graph for $Z(n,d)$ is isomorphic to the Hasse diagram for the higher Bruhat order graded poset $B(n,d)$. We will not bother to define the higher Bruhat order, rather, we will state the relevant results about it in the language of fine zonotopal tilings. Flips in zonotopal tilings correspond to covering relations in $B(n,d)$, and the functional $\phi$ used in \cite[Proposition 3.14]{HenSpe10} on tilings can be related to the rank function on $B(n,2)$.

\begin{theorem}[{\cite[Theorem 4.1]{Zie91}}]\label{bruhat}
The edges of the flip graph for $Z(n,d)$ form the Hasse diagram for a graded poset with unique minimal and maximal elements $\mn$ and $\mx$ at ranks $0$ and $\binom{n}{d+1}$. The set of minimal-length paths of flips between $\mn$ and $\mx$ modulo commutation of unrelated flips is in natural bijection with the elements of $Z(n,d+1)$, such that flips in tilings of $Z(n,d+1)$ swap the order in which $d+2$ flips occur in the corresponding path. 
\end{theorem}

It follows from the above that $Z(d+2,d+1)$ has only two fine zonotopal tilings, so $Z(d+2,d)$ has only two paths from $\mn$ to $\mx$ up to commutation. There are also no pairs of commuting flips in tilings of $Z(d+2,d)$, so its flip graph must be a single $(2d+4)$-cycle. We are now ready to characterize the cycles in the flip graph for zonotopal tilings.

\begin{theorem}\label{zonotopecycles}
Let $\mb Z_{n,d}$ be the two-dimensional regular CW-complex formed by the flip graph for $Z(n,d)$ with the following 2-cells glued:

\begin{itemize}
\item quadrilaterals, wherever there is a cycle of length four corresponding to commuting pairs of flips;
\item $(2d+4)$-gons, wherever there is a cycle of length $(2d+4)$ whose vertices are all refinements of a particular zonotopal tiling which is fine except for a single signed subset which creates a tile isomorphic to $Z(d+2,d)$.
\end{itemize}

Then $\mb Z_{n,d}$ is simply connected.
\end{theorem}

\begin{proof}
We will use a technique similar to the proof in~\cite[Proposition 3.14]{HenSpe10}, and use results about the higher Bruhat order as a black box to generalize to higher dimensions.

Let $\gamma = S_1S_2\cdots S_m$, where each $S_i$ is a flip which turns tiling $\Delta_i$ into $\Delta_{i+1}$ and $\Delta_1 = \Delta_{m+1}$, be a loop in the flip graph for $Z(n,d)$ which connects the tilings $\Delta_1,\Delta_2,\ldots,\Delta_{m+1} = \Delta_1$. It suffices to show that $\gamma$ can be continuously deformed to a point in $\mb Z_{n,d}$. All we know is that the squares and the cycles corresponding to the $(2d+4)$-gon from copies of $Z(d+2,d)$ are nullhomotopic, so our only tool is to replace paths in $\gamma$ with their complement in a square or $(2d+4)$-gon.

First suppose that $\gamma$ is a loop of length $2\binom{n}{d+1}$ that includes $\mn$ and $\mx$. 
Since $\gamma$ connects the minimal and maximal elements twice in the shortest possible time, it can be divided into two parts, $\alpha$ and $\beta$, each of which is a series of monotonic in terms of rank flips in $Z(n,d)$. Then by Theorem~\ref{bruhat}, $\alpha$ and $\beta$ are each representative elements of some equivalence classes of paths between $\mn$ and $\mx$ given by fine zonotopal tilings $A$ and $B$ of $Z(n,d+1)$, respectively. The flip graph for $Z(n,d+1)$ is connected, so there exists a sequence of flips to transform $A$ into $B$. Along the way, commutation of flips in $\alpha$ is required to get the right representative element of $A$, to allow the flips in $Z(n,d+1)$ to be realized as $(2d+4)$-gons in $\mb Z_{n,d}$. The flips in $Z(n,d+1)$ involve $d+2$ tiles in a copy of $Z(d+2,d+1)$, which appear as $d+2$ flips in $\alpha$, all inside a copy of $Z(d+2,d)$. Therefore commutation of flips moves $\alpha$ over a quadrilateral, while flips of $A$ involve moves $\alpha$ over a $(2d+4)$-gon. At each step, a continuous deformation of $\alpha$ occurs, eventually transforming it to $\beta$, at which point $\gamma$ is trivial because it is $\beta\beta^{-1}$.

Now suppose $\gamma$ is any arbitrary loop as before. Then for each vertex $\Delta_i$ in $\gamma$, draw a path $\delta_i$ of length $\binom{n}{d+1}$ between $\mn$ and $\mx$ which goes though $\Delta_i$, using Theorem~\ref{bruhat}. Let's say that $\delta_i = \delta_i^-\delta_i^+$, where $\delta_i^+$ connects $\mn$ to $\Delta_i$, 
and then $\delta_i^-$ connects $\Delta_i$ to $\mx$, both in the shortest possible time. Suppose that the loops $S_i\delta_{i+1}^+(\delta_i^+)^{-1}$ 
are all deformable to point. Then after a continuous deformation we could compute to conclude the result (the brackets denote the homotopy class of a curve with fixed endpoints):
\[[\gamma] = \prod\limits_{i=1}^m [\delta_{i}^+][\delta_{i+1}^+]^{-1} = [\delta_1^+]\left(\prod\limits_{i=2}^m [\delta_{i}^+]^{-1}[\delta_i^+]\right)[\delta_{m+1}^+]^{-1} = [\delta_1^+][\delta_{m+1}^+]^{-1} = 1.\]
Each flip $S_i$ is either an upward flip or a downward flip, depending on whether $\Delta_{i+1}$ has a higher or lower rank than $\Delta_i$ when seen in the higher Bruhat order. If it is an upward flip, then $\delta_i^-(\delta_i^-)^{-1}S_i\delta_{i+1}^+(\delta_i^+)^{-1}$ is a cycle of minimal length which includes $\mn$ and $\mx$, and so is trivial in $\mb Z_{n,d}$ as we showed in the previous paragraph. If it is a downward flip, then $S_i(\delta_i^-)^{-1}\delta_i^-\delta_{i+1}^+(\delta_i^+)^{-1}$ 
is similarly trivial. In either case, the new loop is certainly homotopic to $S_i\delta_{i+1}^+(\delta_i^+)^{-1}$, completing the proof.
\end{proof}

This result will be crucial for an alternate proof of our result for plabic graphs, which is described in the appendix. 

\section{Plabic Graphs in Zonotopal Tilings}\label{plabic}

Let $G$ be an embedding of a planar graph in a disk with each vertex colored white or black (adjacent vertices need not be different colors). Also add $n$ black \emph{boundary vertices} $b_1,\ldots,b_n$ in clockwise order outside of the disk, each with a single edge to one of the vertices of $G$. This configuration is called a \emph{plabic graph} and we refer to it by $G$ (see Figure~\ref{plabicgraphpng}).

\begin{definition}\label{strand}
A \emph{strand} $s_i$ in a plabic graph $G$ is a path which starts at $b_i$, and proceeds along the edges of $G$ until it reaches some boundary vertex $b_j$, according to the \emph{rules of the road}; when $s_i$ reaches a white (resp. black) vertex $v$ through edge $e$, it makes a \emph{sharp} left (resp. right) turn. That is, if the edges of $v$ are shown in a circle, then $s_i$ should traverse the next edge clockwise (resp. counterclockwise) of $e$. The \emph{strand permutation} (or the \emph{connectivity}) of $G$ is the permutation $\pi_G \in S_n$ such that if $s_i$ ends at $b_j$ then $\pi_G(i) = j$.
\end{definition}

We will only deal with a special class of plabic graph. 
A \emph{bad double crossing} is when two distinct strands both traverse edge an $e_1$ and later traverse another edge $e_2$.

\begin{definition}[cf. {\cite[Theorem 13.2]{Pos06}}]\label{Reduced}
A plabic graph $G$ is \emph{reduced} if and only if \begin{itemize}
\item For any edge $e$ between non-boundary vertices, exactly two distinct strands $s_i$ and $s_j$ traverse $e$.
\item $G$ does not contain any bad double crossings.
\item When $\pi_G(i) = i$, the vertex $b_i$ is connected to a single isolated vertex of $G$.
\end{itemize}
\end{definition}

\begin{figure}[ht]\centering
\includegraphics[width=0.60\textwidth]{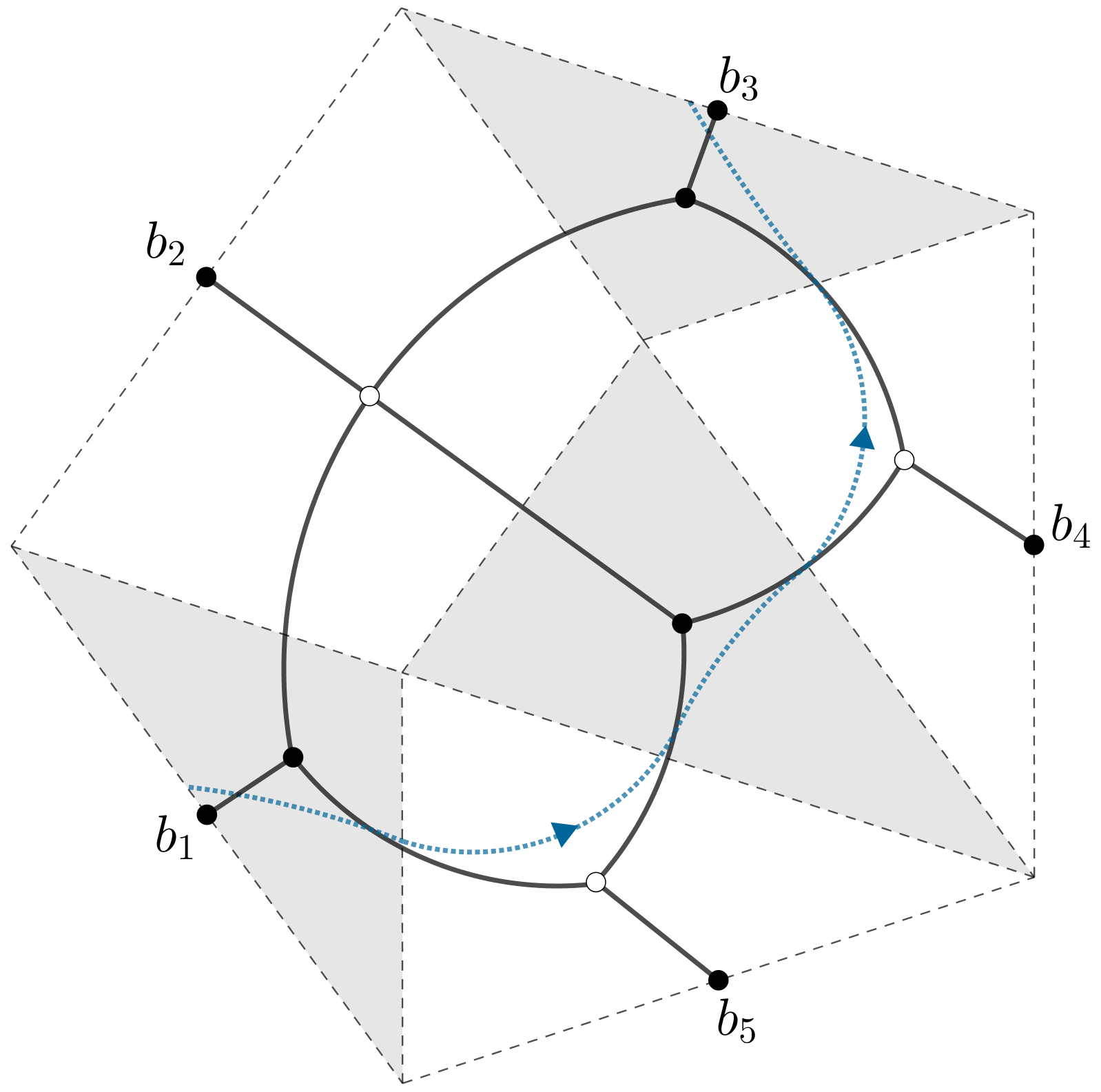}
\caption{A plabic graph with connectivity $\pi{(5,2)}$. The strand $s_1$ is also shown
}
\label{plabicgraphpng}
\end{figure}

\noindent We will only be considering reduced plabic graphs, and so will often omit the word ``reduced''.

The \emph{decorated strand permutation} $\pi_G^:$, for $G$ a reduced plabic graph, is identical to $\pi_G$ except that the fixed points of $\pi_G$ are \emph{decorated} (black) if the single isolated vertex they are connected to is black, otherwise they are \emph{undecorated} (white).

Postnikov described how the \emph{boundary measurements} for reduced plabic graphs parameterize certain ``positroid cells'' in the totally non-negative Grassmannian $\Gr^{\geq 0}(n,k)$~\cite[Thm. 12.7]{Pos06}. The positroid cell parameterized by a plabic graph depends only on its decorated strand permutation. The cyclic permutation which sends $i$ to $i + k$ (modulo $n$) corresponds to the top cell of $\Gr^{\geq 0}(n,k)$, which is the totally positive Grassmannian $\Gr^{>0}(n,k)$ and so is of special interest; we refer to this cyclic permutation by $\pi(n,k)$.

When Postnikov~\cite{Pos06} introduced plabic graphs, he gave some moves to relate them (they correspond to certain simple reparametrizations of the positroid cell). One can check that the moves in Figure~\ref{plabicmovespng} preserve the strand connectivity and whether the plabic graph is reduced. Plabic graphs where all vertices have degree three are called \emph{trivalent}, and $(M1),(M3)$ are called trivalent moves. Through uncontraction moves, any plabic graph can be made trivalent. If all possible contraction moves are performed, the resulting graph will be bipartite.

\begin{theorem}[Postnikov~\cite{Pos06}]\label{connected}
Any two reduced trivalent plabic graphs with the same connectivity can be related by a sequence of the moves (M1), (M2), and (M3) in Figure~\ref{plabicmovespng}.
\end{theorem}

A slightly more general theorem on the flip-connectedness of partly frozen plabic graphs was proven in~\cite{OhSpe14}.

We will primarily deal with trivalent plabic graphs and so only use (M1)--(M3), but the contraction/uncontraction moves will be relevant for triple crossing diagrams.

\begin{figure}[ht]\centering
\includegraphics[width=0.85\textwidth]{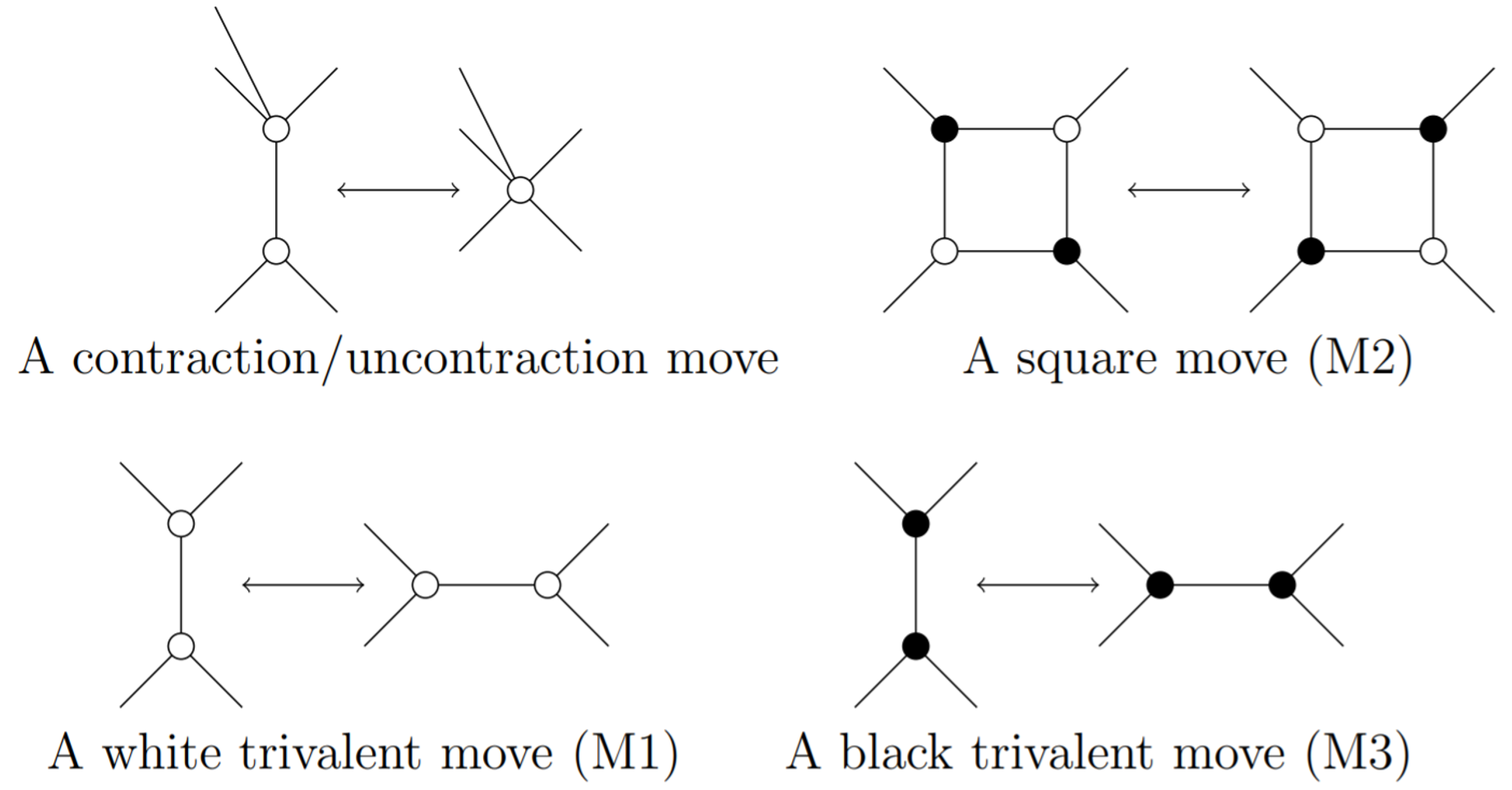}
\caption{Moves in plabic graphs, from Galashin~\cite{Gal17}}
\label{plabicmovespng}
\end{figure}

Pavel Galashin~\cite{Gal17} shows that the $k$-th cross-section, $1 \le k \le n-1$, of fine tilings of the three dimensional cyclic zonotope correspond to trivalent reduced plabic graphs with connectivity $\pi(n,k)$. Let $\Delta$ be a fine zonotopal tiling of $Z(n,3)$. The cross-section $\Sigma_k$ of $\Delta$ with the plane $\{(k,x_2,x_3) \in \mathbb{R}^3\}$ is a triangulation of an $n$-gon, possibly with some interior vertices. We call it a \emph{plabic triangulation}. The vertices of $\Sigma_k$ are labeled by strings in $\{+,-\}^n$ with exactly $k$ `$+$' symbols, or equivalently, by elements of $\binom{[n]}{k}$. For any triangle in $\Sigma_k$, either the union of the labels of the vertices has $k+1$ elements, or the intersection of the labels of the vertices has $k-1$ elements, depending on the location of the triangle has a cross-section of a single parallelepiped tile. In the first case consider the triangle to be \emph{black}, in the second case consider it \emph{white}. Let $G_k$ be the planar dual to the triangulation $\Sigma_k$, and color the vertices of $G_k$ according to the color of the triangle to which it belonged.

\begin{theorem}[Galashin~\cite{Gal17}]\label{galashin}
$G_k$ is a trivalent reduced plabic graph with strand connectivity $\pi(n,k)$. Further, for any trivalent reduced plabic graph $G$ with strand connectivity $\pi(n,k)$, there exists a fine zonotopal tiling of $Z(n,3)$ for which $G_k = G$.
\end{theorem}

If we erase all the edges in $\Sigma_k$ between the regions of the same color, we get what is called a \emph{plabic tiling}. The planar dual of a plabic tiling is a bipartite plabic graph. The vertices of $\Sigma_k$ appear in the faces of $G_k$, and so we will refer to their labels as the \emph{face labels} of $G_k$.

We would like to see how these trivalent plabic moves relate to the three-dimensional cyclic zonotopal flips. Galashin~\cite{Gal17} observed that a zonotopal flip at height $k$ performs a square move in $G_k$, a white trivalent move in $G_{k-1}$, and a black trivalent move in $G_{k+1}$.

\begin{lemma}\label{squareflips}
For any zonotopal tiling $\Delta$ of $Z(n,3)$, the available flips are in bijective correspondence with the available square moves in the plabic graphs $\{G_k\}_{k=1}^n$.
\end{lemma}
\begin{proof}
A proof does not completely appear in~\cite{Gal17}, so we will include one here. Take any available flip $S \in \binom{[n]}{4}$ for $\Delta$, say $S = \{a,b,c,d\}$ with $a < b< c< d$. Let $k:= |S_a^+| + 2$. Then the intersection $\tau_{X_a} \cap \tau_{X_b} \cap \tau_{X_c} \cap \tau_{X_d} =: v$ is a vertex in $\Delta$ which is in the cross-section $\Sigma_{k}$. Further, the cross-sections of the tiles $\tau_{X_i}$ at level $k$ are triangles in $\Sigma_k$ which include $v$ as a vertex. The color of the triangle corresponding to $X_i$ is determined by the whether the plane $x = k$ cuts the tile at a lower or higher part, and so depends only on the value of $s_i$. Then the triangles from $X_a,X_c$ are of one color and $X_b,X_d$ have the other color, by Definition~\ref{zonotopeflip}. Finally, the intersections $\tau_{X_a}\cap \tau_{X_b},\tau_{X_b}\cap \tau_{X_c},\tau_{X_c}\cap \tau_{X_d},\tau_{X_d}\cap\tau_{X_a}$ all appear as edges connected to $v$ in $\Sigma_k$, because each of these tiles is a quadrilateral with two vertices at height $k$, one of which is $v$. Therefore $G_k$ has a square move pattern formed by the vertices in the four vertices from the four tiles. Performing this flip performs this square move and no other square moves in any other layer. 

It now suffices to invert this map. That is, take any available square move in any layer $G_k$, and recover the unique flip which performs that square move. Well, the square move is formed by four triangles in $\Sigma_k$, whose five vertices, when considered as strings in $\{+,-\}^n$, agree in all but four coordinates. This can be seen by noting that all five strings have exactly $k$ `$+$' symbols, and that when two vertices are adjacent they can only differ in two coordinates. These four coordinates $a < b < c < d$ form our set $S$, and if the flip corresponding to $S$ is available, then it must correspond to this square move in the map described in the previous paragraph. It then suffices to check that $S$ satisfies the conditions in Definition~\ref{zonotopeflip}. Indeed, $S_a^+ = S_b^+ = S_c^+ = S_d^+$, because the vertices agree on all coordinates outside of $S$. Now, of the four outer vertices, two are white and two are black, so two of $a,b,c,d$ will have $s_i = 1$ and two will have $s_i = 0$. Moreover, these colors are oriented in a cyclically alternating fashion, and they also correspond to the signed subsets $X_a,X_b,X_c,X_d$ in a cyclic fashion. Therefore we must have $X_a = X_c \neq X_b = X_d$, so we can conclude that $S$ is an available flip in $\Delta$.
\end{proof}

We would like to know when the other two plabic moves can be performed as well. A white or black trivalent move depends on the existence of a square move in a neighboring layer, so the following result about the relationship between the graphs $G_k$ is helpful.

\begin{lemma}[Galashin~\cite{Gal17}]\label{compatible}
Let $\Sigma_k$ be a colored and labeled triangulation for some tiling $\Delta$. Then $\Sigma_{k+1}$ is fixed up to the triangulation of the white regions and $\Sigma_{k-1}$ is fixed up to the triangulation of the black regions.
\end{lemma}\begin{proof}
By Galashin's~\cite[Corollary 4.4]{Gal17}, the vertex labels of $\Sigma_{k+1}$ and $\Sigma_{k-1}$ are completely determined by $\Sigma_k$. The white triangles in $\Sigma_k$ cut a tile of $\Delta$ which is cut by a black triangle in $\Sigma_{k+1}$, and all black triangles in $\Sigma_{k+1}$ correspond to a white triangle in $\Sigma_k$. Similarly, the black triangles in $\Sigma_k$ give the white triangles in $\Sigma_{k-1}$. Therefore all the white and black regions are determined in both $\Sigma_{k+1}$ and $\Sigma_{k-1}$, and indeed all that is left is the triangulation of the white regions in $\Sigma_{k+1}$ and the black regions in $\Sigma_{k-1}$.
\end{proof}

This allows us to define $\UP$ and $\DOWN$ operations on plabic graphs with cyclic strand connectivity.

\begin{definition}\label{updown}
Let $G$ be a trivalent plabic graph with strand connectivity $\pi(n,k)$, and let $\Delta$ be any fine zonotopal tiling of $Z(n,3)$ such that $\Sigma_k$ is dual to $G$. Then $\UP(G,\Delta)$ is the plabic graph dual to $\Sigma_{k+1}$, and $\DOWN(G,\Delta)$ is the plabic graph dual to $\Sigma_{k-1}$. If $\Delta$ is not specified, $\UP(G)$ and $\DOWN(G)$ can be computed with an arbitrary valid $\Delta$, and the corresponding plabic triangulations are determined up to white and black triangulation, respectively.
\end{definition}

By Lemma~\ref{compatible}, we can determine $\UP(G)$ up to white trivalent moves, and $\DOWN(G)$ up to black trivalent moves. In particular, the face labels are determined exactly. In the next section, we will extend the definition of $\UP$ and $\DOWN$ to plabic graphs with any strand connectivity, and use these operations to prove our main result. The same operations can be applied to plabic triangulations, or even plabic tilings (and non-trivalent plabic graphs) if an arbitrary triangulation is chosen.

\section{Cycles in the Plabic Flip Graph}\label{cycles}

For a given decorated permutation $\pi^:$, any two trivalent reduced plabic graphs with connectivity $\pi^:$ can be related by a sequence of the moves (M1)--(M3). The \emph{flip graph} $F_{\pi^:}$ is the graph whose vertices are trivalent reduced plabic graphs with connectivity $\pi^:$ and whose edges connect plabic graphs related by a move. Cycles in the flip graph correspond to sequences of moves which leave the plabic graph unchanged. We are going to use the $\UP$ and $\DOWN$ operations to study these cycles; since we want our proof to apply to plabic graphs of any connectivity, we want to be able to include any plabic graph inside one with the cyclic connectivity. To do this we need to introduce some more terminology.




\begin{definition}[{\cite[Definition 16.1]{Pos06}}]\label{necklace}
A \emph{Grassmann necklace} $\mc I = (I_1,I_2,\ldots,I_n)$ is a sequence of subsets of $[n]$ of the same size such that for all $i$ there exists $j$ such that $I_{i+1} = (I_i\setminus \{i\})\cup \{j\}$, where as an index $i$ is considered modulo $n$.
\end{definition}

Grassmann necklaces are in bijection with decorated permutations via \emph{juggling patterns} (see~\cite[Section~3]{KnuLamSpe2013}) of period $n$ and throws of height at most $n$, and all of these are in correspondence with \emph{positroids} inside $\binom{[n]}{k}$ (originally shown directly in~\cite{Pos06}). Let $\phi$ be the bijection sending necklaces to permutations, and call the size of any element of $\phi^{-1}(\pi^:)$ the \emph{helicity} of $\pi^:$. 

\begin{definition}\label{grassupdown}
Suppose $\mc I = (I_1,\ldots,I_n)$ is a Grassmann necklace such that $\phi(\mc I)$ is not an identity decorated permutation. Define $\DOWN(\mc I)$ to be the necklace $(I_1\cap I_{\iota(1)},I_2\cap I_{{\iota(2)}},\ldots,I_n\cap I_{{\iota(n)}})$, where $\iota(j)$ is the last index before $j$  (in the cyclic order) such that $I_j \neq I_{\iota(j)}$. Similarly, $\UP(\mc I)$ is the necklace $(I_1\cup I_{\lambda(1)},I_2\cup I_{{\lambda(2)}},\ldots,I_n\cup I_{{\lambda(n)}})$, where $\lambda(j)$ is the first index after $j$ (in the cyclic order) such that $I_j \neq I_{\lambda(j)}$.
For decorated permutations $\pi^:$, define $\DOWN(\pi^:) \coloneqq \phi(\DOWN(\phi^{-1}(\pi^:)))$ and $\UP(\pi^:) \coloneqq \phi(\UP(\phi^{-1}(\pi^:)))$.
\end{definition}

Let us check that $\DOWN(\mc I)$ and $\UP(\mc I)$ are well-defined. Since $\phi(\mc I)$ is not an identity, not all the $I_j$ coincide, and the functions $\iota(\cdot)$ and $\lambda(\cdot)$ are well-defined. If $|I_j| = k$ for all $j$ then clearly $|I_j\cap I_{\iota(j)}| = k-1$ for all $j$ and $|I_j\cup I_{\lambda(j)}| = k+1$ for all $j$. How do $I_j \cap I_{\iota(j)}$ and $I_{j+1} \cap I_{\iota(j+1)}$ differ? Note that 
$$
(I_j \cap I_{\iota(j)}) \setminus (I_{j+1} \cap I_{\iota(j+1)}) = (I_j \cap I_{\iota(j)}) \setminus I_{j+1} \subseteq I_j \setminus I_{j+1} \subseteq \{j\},
$$
so $\DOWN(\mc I)$ is indeed a Grassmann necklace.
How do $I_j \cup I_{\lambda(j)}$ and $I_{j+1} \cup I_{\lambda(j+1)}$ differ? Note that 
$$
(I_j \cup I_{\lambda(j)}) \setminus (I_{j+1} \cup I_{\lambda(j+1)}) = I_j \setminus (I_{j+1} \cup I_{\lambda(j+1)}) \subseteq I_j \setminus I_{j+1} \subseteq \{j\},
$$
so $\UP(\mc I)$ is indeed a Grassmann necklace. 



Now we state a useful result of Oh, Postnikov, and Speyer~\cite{Oh15} (for our purposes it does not matter what positroids are, or what it means for a collection to be weakly separated).

\begin{theorem}[{\cite[Theorems 1.3 and 1.5]{Oh15}}]\label{weak}
Suppose $\mc M \subset \binom{[n]}{k}$ is a positroid, $\mc I$ is the corresponding Grassmann necklace, and $\pi^: \coloneqq \phi(\mc I)$ the corresponding decorated permutation. Then the maximal by inclusion weakly separated collections inside $\mc M$ that contain $\mc I$ are exactly the collections of face labels for the plabic graphs with connectivity $\pi^:$.
\end{theorem}

The positroid corresponding to the top cell $\Gr^{>0}(n,k)$ is all of $\binom{[n]}{k}$ and the corresponding decorated permutation is $\pi(n,k)$. A plabic graph is completely determined by its face labels. Therefore plabic graphs for the cyclic permutations are the maximal by inclusion weakly separated collections in $\binom{[n]}{k}$. Then if $G$ is any plabic graph, then its collection of face labels is a subset of those for some plabic graph $G'$ for the cyclic permutation. 

Since Grassmann necklaces and decorated permutations are in bijection, we can talk about the class of plabic graphs for a certain Grassmann necklace instead. If a plabic graph $G$ has strand permutation $\pi^:$, and the corresponding necklace is $\mc I = \phi^{-1}(\pi^:)$, then the face labels of the perimeter regions of the graph form $\mc I$ as well. To summarize: Theorem~\ref{weak} implies that any plabic graph $G$ with the strand permutation $\pi^:$ of helicity $k$ can be embedded as a subgraph into another plabic graph $G'$ corresponding to the cyclic permutation $\pi(n,k)$. Let $\Sigma_k$ be a plabic triangulation dual to $G'$, constructed as in Section~\ref{plabic}. The Grassmann necklace $\mc I$, considered as a polygonal curve in $\Sigma_k$, encloses triangulation $\sigma$, which we call the \emph{plabic triangulation} for $G$. If $G$ was trivalent, $\sigma$ is defined unambiguously; otherwise, $\sigma$ is defined up to triangulating some of its convex monochromatic regions, corresponding to the vertices of $G$ of degree greater than three. We regard $\sigma$ as a geometric realization of $G$ in the plane $\{(k,x_2,x_3) \in \mathbb{R}^3\}$. An equivalent realization of plabic graphs first appeared in~\cite[Section~9]{Oh15}, while its relation to zonotopal tilings was articulated in~\cite{Gal17}.

Note that a Grassmann necklace, embedded into the plane as above, is a closed polygonal, possibly self-intersecting curve. However, the self-intersections are never transversal, so after an ``infinitesimal'' perturbation, the necklace can be viewed as a closed Jordan curve, enclosing a region to be triangulated. Therefore, plabic triangulations may have degeneracies. For example, the plabic triangulation corresponding to the necklace $(134, 234, 134, 145, 135)$ consists of a black triangle together with a hanging edge; the necklace $(1245,2345,1345,1245,1258,1268,1278,1258)$ bounds the region consisting of two triangles connected by an edge (see Figure~\ref{selfintersect}). If the decorated permutation corresponding to a necklace is an identity then the triangulation consists of a single vertex; if the corresponding decorated permutation is a transposition then the triangulation is just an edge. 

\begin{figure}[ht]\centering
    \includegraphics[width=0.55\textwidth]{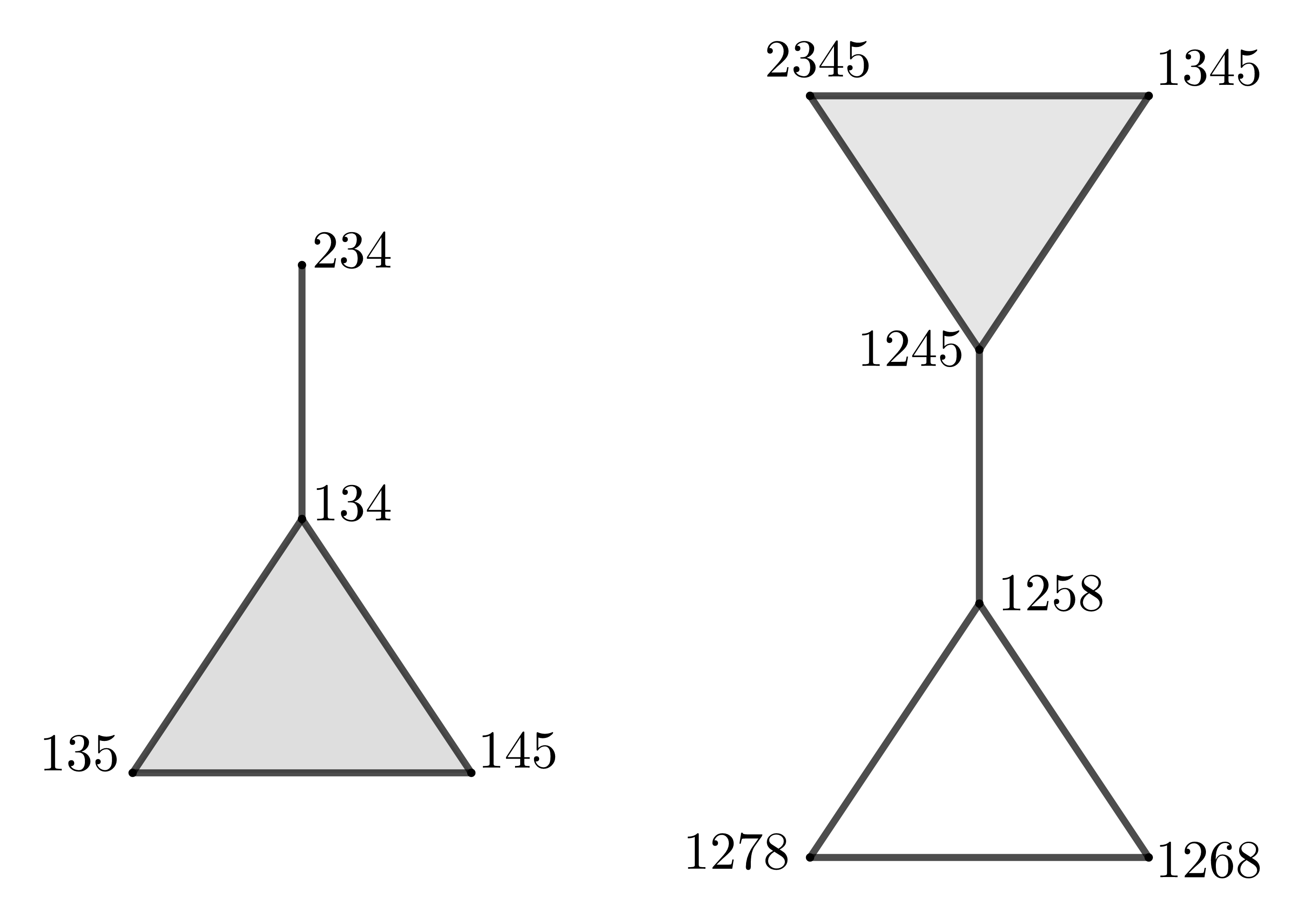}
    \caption{Grassmann necklaces which are not Jordan curves}
    \label{selfintersect}
\end{figure}


We can use the above remarks to generalize Definition~\ref{updown}. Let $\Delta$ be a fine zonotopal tiling of $Z(n,3)$ inducing the triangulation $\Sigma_k$ in the $k$-th layer, and let $I_1$, $I_2$ be the face labels of any two adjacent vertices in $\Sigma_k$. It follows that $I_1 = S \cup \{a\}, I_2 = S \cup \{b\}$ for some $S \in \binom{[n]}{k-1}$, $a \neq b \in [n]$. Then the vertex with label $I_1 \cap I_2$ is also present in the tiling $\Delta$, as well as the vertex $I_1 \cup I_2$. If the edge $I_1 - I_2$ is an edge of a necklace $\mc I$ then $I_1 \cap I_2 \in \DOWN(\mc I)$ and $I_1 \cup I_2 \in \UP(\mc I)$. Therefore, if the tiling $\Delta$ respects $\mc I$ then it respects both $\DOWN(\mc I)$ and $\UP(\mc I)$. 

\begin{definition}\label{updowngeneral}
Let $G$ be a plabic graph with non-identity strand permutation $\pi^:$ of helicity $k$, and let $\Delta$ be a fine zonotopal tiling of $Z(n,3)$ for which the dual $G'$ of $\Sigma_{k}$ contains $G$ as a subgraph. Then define $\UP(G,\Delta)$ to be the subgraph of $\UP(G',\Delta)$ surrounded by $\UP(\phi^{-1}(\pi^:))$, and $\DOWN(G,\Delta)$ to be the subgraph of $\DOWN(G',\Delta)$ surrounded by $\DOWN(\phi^{-1}(\pi^:))$. If $\Delta$ is not specified, $\UP(G)$ and $\DOWN(G)$ are determined up to white and black trivalent moves, respectively.
\end{definition}

\begin{lemma}\label{updownnecklace}
Let $\mc I = \phi(\pi^:)$ be a necklace such that $\pi^:$ is not an identity and $\DOWN(\pi^:)$ is not an identity. Then the necklace $\mc J \coloneqq \UP(\DOWN(\mc I))$ is nested in $\mc I$, when embedded geometrically in the plane. 
\end{lemma}


\begin{proof}
Let $\mc I = (I_1,\ldots,I_n)$ be such a necklace, let $\mc L = (L_1,\ldots,L_n)$ be $\DOWN(\mc I)$, and let $\mc J = (J_1,\ldots,J_n)$ be $\UP(\mc L)$. Say that $\iota(j)$ denotes the last index cyclically before $j$ such that $I_j \neq I_{\iota(j)}$, and $\lambda(j)$ denotes the next index cyclically after $j$ such that $L_j \neq L_{\lambda(j)}$. Note that $\iota$ and $\lambda$ are well-defined because $\phi^{-1}(\mc I)$ and $\phi^{-1}(\mc L)$ are both not an identity, so not all of the $I_j$ and $L_j$ are the same. Our notation here gets at the same data as the \emph{reduced Grassmann necklaces} of Farber and Galashin~\cite[Section 6]{FarGal18}. By Definition~\ref{updowngeneral}, we have 
$$
J_j = L_j \cup L_{\lambda(j)} = (I_{j}\cap I_{\iota(j)})\cup (I_{\lambda(j)} \cap I_{\iota(\lambda(j))}).
$$ 

We claim that $J_j = I_{\iota(\lambda(j))}$. Note that $\iota(\lambda(j))$ belongs to the half-open interval $[j, \lambda(j))$ in the cyclic order. 
Therefore, $L_j = L_{\iota(\lambda(j))}$. Now we notice that $L_j = L_{\iota(\lambda(j))} = I_{\iota(\lambda(j))} \cap L_{\iota(\iota(\lambda(j)))} \subseteq I_{\iota(\lambda(j))}$, and $L_{\lambda(j)} = I_{\lambda(j)} \cap I_{\iota(\lambda(j))} \subseteq I_{\iota(\lambda(j))}$; hence, $J_j = L_j \cup L_{\lambda(j)} \subseteq I_{\iota(\lambda(j))}$. But $J_j$ and $I_{\iota(\lambda(j))}$ are of the same size, so they coincide.

We have now shown that $J_j = I_{\iota(\lambda(j))}$ for all $j$; hence the vertices of the polygonal curve $\mc J$ form a subset of the vertices of $\mc I$. Let $J_i - J_{j+1}$ be a non-degenerate edge of $\mc J$, that is, $J_j \neq J_{j+1}$. Then $\lambda(j) = j+1$ and $\iota(j+1) = j$. From the definitions of $\iota(\cdot)$ and $\lambda(\cdot)$ it follows that $J_j =  I_{\iota(\lambda(j))} = I_j$ and that $j = \iota^m(\lambda(j+1))$ for some $m \ge 2$. If $m = 2$, the edge $J_i - J_{j+1}$ coincides with the edge $I_j - I_{j+1}$ of $\mc I$. If $m \ge 3$, take a closer look at the polygonal curve $I_{\iota^m(\lambda(j+1))} - I_{\iota^{m-1}(\lambda(j+1))} - \ldots - I_{\iota(\lambda(j+1))}$. The common intersection of all those sets is $L_{\iota^{m-1}(\lambda(j+1))} = \cdots = L_{\iota(\lambda(j+1))}$, so this is a convex curve. Together with the edge $J_i - J_{j+1}$ it bounds a region, which has no other $I_i$'s inside and which is always white in any plabic triangulation of the domain inside $\mc I$. Therefore, the curve $\mc J$ can be obtained from the curve $\mc I$ by replacing some of its convex parts by straight edges. Thus, $\mc J$ is nested in $\mc I$. 
\end{proof}

\begin{remark}
As the referee pointed out, the proof of Lemma~\ref{updownnecklace} in fact describes the following effect of the composition $\UP \circ \DOWN$ on a necklace $\mc I = (I_1, \ldots, I_n)$: each $I_j$ is replaced with $I_{\alpha(j)}$, where $\alpha(j)$ is the first index $k$ after $j-1$ (in the cyclic order) such that $\phi(\mc I)(k-1) \neq k$.
\end{remark}

We are now ready to prove our main result.

\begin{theorem}\label{plabicycles}
Let $\mb X_{\pi^:}$ be the two-dimensional regular CW-complex given by the flip graph of trivalent reduced plabic graphs with connectivity $\pi^:$, with the following $2$-cells glued to it (cf. Figure~\ref{plabicyclepng}):
\begin{itemize}
\item A quadrilateral, wherever there is a 4-cycle generated by two moves occurring in separate parts of a plabic graph;
\item A pentagon, wherever there is a 5-cycle generated by five white or five black trivalent moves, such that all flips take place in a subgraph which forms a plabic graph with connectivity $\pi{(5,1)}$ or $\pi{(5,4)}$ (we call those white and black pentagonal cells, respectively);
\item A decagon, wherever there is a 10-cycle consisting of 5 plabic moves alternating with 5 white or 5 black trivalent moves, such that all flips in the cycle take place in a subgraph which forms a plabic graph with connectivity $\pi{(5,2)}$ or $\pi{(5,3)}$ (we call those white and black decagonal cells, respectively).
\end{itemize}
Then $\mb X_{\pi^:}$ is simply connected. 
\end{theorem}

We need to show that every loop $L$, which is a continuous map $S^1 \to \mb X_{\pi^:}$, there is a continuous map $f: D^2 \to \mb X_{\pi^:}$ from the two-dimensional disk to $\mb X_{\pi^:}$, whose restriction $f\vert_{S^1}$ onto the boundary $S^1 = \partial D^2$ is $L$. In fact, we will prove a slightly stronger statement: we will build a contraction $f$ of a special ``cellular'' kind. 

By the standard cellular approximation argument, we can homotope $L$ so that it follows only the edges of $\mb X_{\pi^:}$. So $L$ can be regarded combinatorially as a loop of plabic moves, connecting plabic graphs with necklace $\mc I = \phi(\pi^:)$. 

\begin{definition}\label{cellcontract}
Let $L$ be a loop, possibly self-intersecting, of plabic moves in $\mb X_{\pi^:}$, as above. The following data is said to be a \emph{cellular contraction} of $L$.
\begin{itemize}
    \item A polyhedral complex $\mb D \subset \mb R^2$, consisting of a finite number of points (vertices, or $0$-cells), straight segments (edges, or $1$-cells) and convex 4-, 5-, 10-gons in the plane ($2$-cells of the complex); 
    the intersection between any two closed cells should be either empty or to be a face of each of them. The union of all cells should be a contractible closed set in $\mb R^2$, which we denote $|\mb D|$.
    \item A continuous map $f : |\mb D| \to \mb X_{\pi^:}$ of cell complexes, mapping vertices to vertices, edges to edges, $m$-gonal $2$-cells to $m$-gonal $2$-cells. Moreover, $f$ should be a homeomorphism on each cell of $\mb D$. 
    \item Labels of the cells of $\mb D$, obtained as follows. Each cell $e$ of $\mb D$ (of any dimension) inherits the label corresponding to the cell $f(e)$ of $\mb X_{\pi^:}$. That is, the cells of $\mb D$ are labeled by plabic graphs, plabic moves and 4-/5-/10-cycles of flips just as in $\mb X_{\pi^:}$; of course, the labels may repeat.
    \item A continuous map $\ell : S^1 \to \partial |\mb D|$, traversing the boundary of $|\mb D|$ clockwise so that the composition $f \circ \ell$ coincides with the loop $L$. The boundary $\partial |\mb D|$ might not be a Jordan curve, but it becomes a Jordan curve after a slight inflation of $|\mb D|$, and this is how the word ``clockwise'' should be understood.
\end{itemize}
\end{definition}

Clearly, if $L$ admits a cellular contraction then it is contractible in the usual sense.

\begin{proof}[Proof of Theorem~\ref{plabicycles}]

We induct on the helicity of $\pi^:$, which we denote $k$. The case $k=0$ is trivial, since the only decorated permutation of helicity $0$ is the identity with all elements white, there is just one plabic graph of this strand connectivity. The case $k=1$ is not trivial, but seems well-known. In this case the vertices of $\mb X_{\pi^:}$ are just triangulations of a convex polygon, the flips correspond to the switch of the diagonal in a quadrilateral, and the only non-trivial relation between flips is the 5-cycle of flips in a pentagon. The result follows from the original work of Stasheff~\cite{Sta63}, but in Appendix we include an independent proof of this, involving Theorem~\ref{zonotopecycles}. We emphasize that the proof in Appendix (and apparently, all the older proofs) builds cellular contractions. 

Now assume $k \ge 2$. Let $\mc I$ the Grassmann necklace corresponding to $\pi^:$. 
In the rest of this proof, we use words ``(trivalent) plabic graph'' and ``plabic triangulation'' interchangeably, assuming the bijection between them.

Take any loop $L$ of plabic moves, connecting plabic graphs with necklace $\mc I = \phi(\pi^:)$. If all the moves along $L$ are just white trivalent moves, then the contraction is built as follows. All the moves are triangulation flips in a few convex white regions. Homotoping $L$ across quadrilateral cells we can make $L$ into the wedge of loops associated to separate white regions. Then build cellular contractions for those loops independently across white pentagonal cells, using the base case $k=1$. 

If $\pi^:$ is an identity decorated permutation, there is just one plabic graph of that strand connectivity, so the theorem follows trivially. If $\pi^:$ is not an identity, $\DOWN(\pi^:)$ is well-defined. If $\DOWN(\pi^:)$ happens to be an identity, it follows that $\mc I$ encloses a white region, and this case was already discussed above. So we assume that $\DOWN(\pi^:)$ is not an identity, so that $\mc J \coloneqq \UP(\DOWN(\mc I))$ is well-defined.



\textbf{Step 1.} We define a loop $\DOWN(L)$ of plabic graphs within the necklace $\DOWN(\mc I)$, with strand permutation $\DOWN(\pi^:)$. For each graph $G$ in $L$, the graph $\DOWN(G)$ is determined up to the triangulation of black regions. Choose those triangulations in an arbitrary way (we keep the notation $\DOWN(G)$ for those plabic triangulations). For every edge $G-G'$ in $L$, the graphs $\DOWN(G)$ and $\DOWN(G')$ differ by at most one move corresponding to the move between $G$ and $G'$, plus maybe some black trivalent moves. Connect $\DOWN(G)$ and $\DOWN(G')$ by a chain of plabic graphs through those moves (the chain might be of length zero if the move $G-G'$ was white trivalent and $\DOWN(G) = \DOWN(G')$). All the chains, over all edges $G-G'$ of $L$ can be incorporated into a loop of plabic graphs with necklace $\DOWN(\mc I)$. We call this loop $\DOWN(L)$. Note that this loop consists of more than one vertex, because $L$ contained some moves other from white trivalent ones.

All those graphs in $\DOWN(L)$ have helicity $k-1$, so we can apply the inductive hypothesis to contract it in the complex $\mb X_{\DOWN(\pi^:)}$. Moreover, the induction hypothesis gives us a cellular contraction $f: \mb D \to \mb X_{\DOWN(\pi^:)}$ of the loop $\DOWN(L)$, as in  Definition~\ref{cellcontract}. In particular, there is a map $\ell : S^1 \to \partial |\mb D|$ with $f \circ \ell = \DOWN(L)$.




\begin{figure}[ht]\centering
\includegraphics[width=0.96\textwidth]{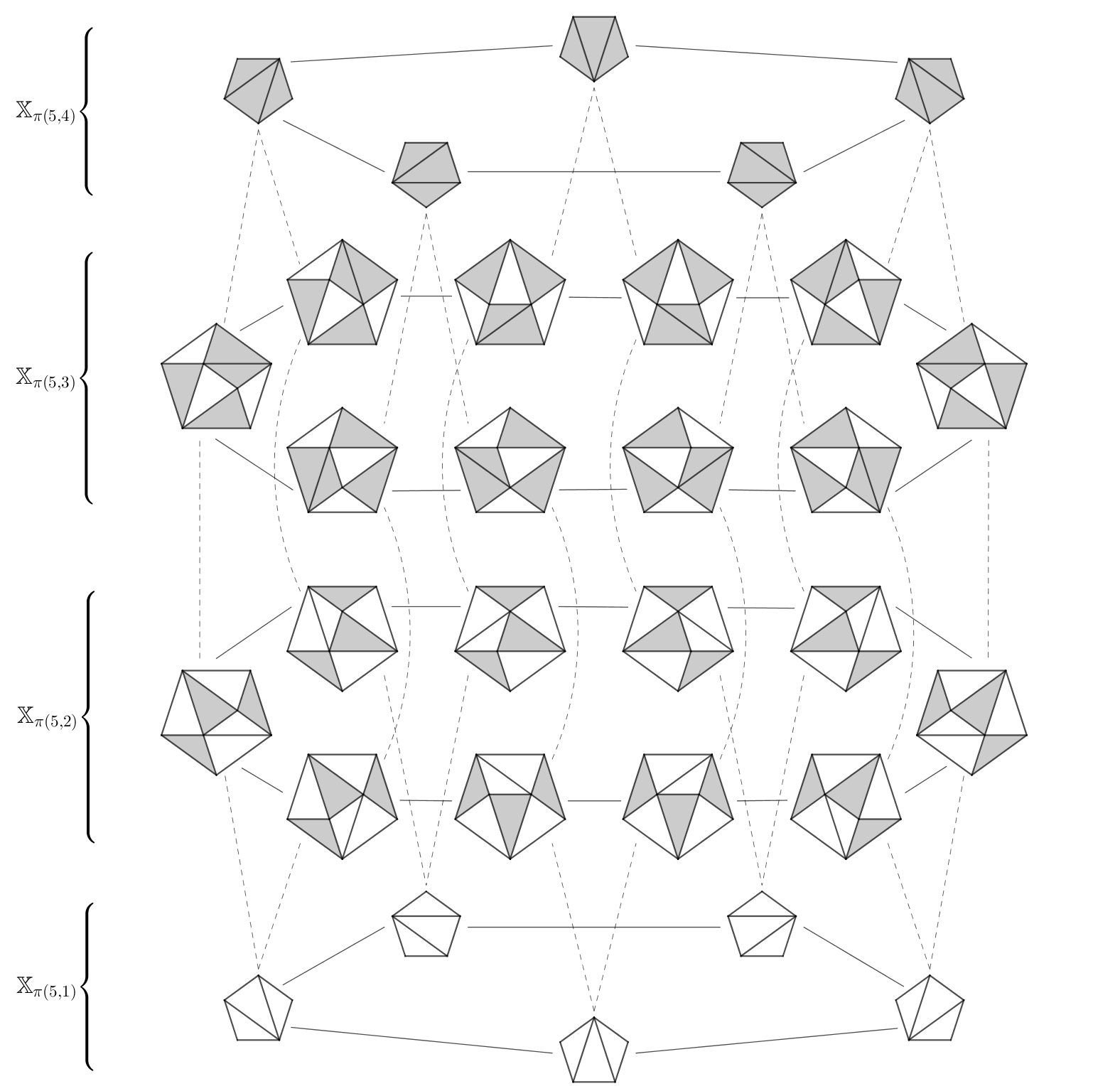}
\caption{Plabic cycles arising from the 10-cycle of zonotopal flips in $Z(5,3)$}
\label{plabicyclepng}
\end{figure}

\textbf{Step 2.} We would like to ``lift'' $f$ back to the $k$-th level, that is, to construct a cellular contraction $\UP(f) : \UP(\mb D) \to \mb X_{\pi^:}$, whose restriction to the boundary would be $L$. We build this contraction one cell at a time.


We start by lifting the polygonal $2$-cells of $\mb D$, which we denote $P_i$.
Each of them can be lifted separately to $\UP(\mb D)$ in the following manner, which can be easily guessed if one looks at the cross-sections of the 10-cycle of the zonotopal tilings of $Z(5,3)$, see Figure~\ref{plabicyclepng}. 

\begin{itemize}
    \item If $P_i$ is a black pentagonal cell, $\UP(P_i)$ is a degenerate polygonal cell, consisting of a single vertex, whose label is the graph, lying above all the five graphs-labels of $P_i$. The triangulation of white regions can chosen in an arbitrary way. 
    \item If $P_i$ is a black decagonal cell, $\UP(P_i)$ is a black pentagonal cell, whose edges correspond to the square moves of $P_i$. The triangulation of white regions should be chosen to be the same for the five graphs-labels of $\UP(P_i)$, but it can be done in an arbitrary way. 
    \item If $P_i$ is a white decagonal cell, then $\UP(P_i)$ is a black decagonal cell. The triangulation of white regions for the labels of $\UP(P_i)$ should be chosen so that all the moves happen in a copy of $\phi(\pi{(5,3))}$-necklace, and outside of it the white regions are triangulated in the same (but arbitrary) manner.
    \item If $P_i$ is a white pentagonal cell, then $\UP(P_i)$ is a white decagonal cell, whose square moves correspond to the moves of $P_i$. The triangulation of white regions should be chosen so that all the moves happen in a copy of $\phi(\pi{(5,2)})$-necklace, and outside of it the white regions are triangulated in the same (but arbitrary) manner.
    \item If $P_i$ is a quadrilateral cell labeled by two commuting moves, then $\UP(P_i)$ is either a quadrilateral, or an edge, or a vertex, depending on how many of those moves were black trivalent. The triangulation of white regions is arbitrary but consistent.
    
\end{itemize}

Now we make an important adjustment to this lifting. All the graphs in the labels of the lifted polygons correspond to the necklace $\mc J = \UP(\DOWN(\mc I))$, which might not coincide with $\mc I$. If it does not, use Lemma~\ref{updownnecklace} to extend those graphs (in the same manner) so that they all correspond to the necklace $\mc I$. 

So far we lifted every $P_i$ separately. Now we glue them to one another and to $L$, using the gluing pattern of $\mb D$. The following steps should be done over all edges and all vertices of $\mb D$. The labels of the cells of $\UP(\mb D)$ will constitute the map $\UP(f)$. 

\begin{enumerate}
    \item Let $e$ be an edge of $\mb D$ not from $\partial |\mb D|$, and let $P$, $Q$ be the two adjacent polygonal $2$-cells. 
    
    \begin{itemize}
        \item If $e$ is (labeled by) a black trivalent move (so that it gets contracted to a vertex after lifting), then we connect the vertices corresponding to $e$ in $\UP(P)$ and $\UP(Q)$ by a chain of white trivalent moves (possibly of zero length).
        \item If $e$ is (labeled by) a square move or a white trivalent move, there are edges corresponding to $e$ in both $\UP(P)$ and $\UP(Q)$, call them $p_1-p_2$ and $q_1-q_2$, respectively. Note that $p_1$ and $q_1$ differ by some white trivalent moves, and $p_2$ and $q_2$ differ by the same exact white trivalent moves. Let $p_1-a_1-...-z_1-q_1$ and $p_2-a_2-...-z_2-q_2$ be chains of white trivalent moves such that every pair $a_1-a_2$, $\ldots$, $z_1-z_2$ is related by the same move as the pairs $p_1-p_2$ and $q_1-q_2$ (this move arises from lifting $e$). This way we connect the edges $p_1-p_2$ and $q_1-q_2$ by a sequence of quadrilateral $2$-cells (possibly of zero length).
    \end{itemize}
    
    \item Let $e$ be a boundary edge from $\partial |\mb D|$, which belongs to a single polygonal $2$-cell $P$. We glue $\UP(P)$ to $L$ following the same scheme as in the previous item.
        If $e$ is a black trivalent move, we connect the vertices corresponding to $e$ in $\UP(P)$ and $L$ by a chain of white trivalent moves.
        If $e$ is a square move or a white trivalent move, there are edges corresponding to $e$ in both $\UP(P)$ and $L$, and we connect them by a sequence of quadrilateral cells in the same way as above.
    
    \item Let $e$ be a boundary edge of $\mb D$ not adjacent to any of the polygons (that is, a hanging edge or a bridge). Then $e$ is bypassed by $\DOWN(L)$ two times. We connect the two corresponding vertices in $L$ by a chain of white trivalent moves (is $e$ was black trivalent), or the two corresponding edges in $L$ by a strip of quadrilateral cells (is $e$ was black trivalent), just as above.
    
    \item Let $v$ be an internal (not lying on the boundary) vertex of $\mb D$. Let $Q_1, \ldots, Q_m$ be the polygonal $2$-cells in $\mb D$ sharing a common vertex $v$, indexed in the order they follow around $v$, so that $Q_i$ and $Q_{i+1}$ (cyclic indexing) share an edge. We already connected the polygons $\UP(Q_i)$ cyclically by strips of quadrilaterals and/or chains of white trivalent moves. Those strips/chains bound a loop whose all vertices match with $\UP(v)$, and differ only by white trivalent moves. Such a loop can be contracted in a cellular way as it follows from the base case $k=1$. That is, we can fill in this loop by a cellular contraction consisting of white pentagonal and quadrilateral cells.
    
     \item Let $v$ be a boundary vertex of $\mb D$. 
        Write down the sequence $Q_1, \ldots, Q_m$ of the adjacent to $v$ polygonal $2$-cells in the order they follow around $v$, including $L$ in this sequence every time the adjacent region is the outside of $|\mb D|$. Then proceed the same way as if $v$ was internal, by filling in a loop of white trivalent moves.
    
\end{enumerate}


The construction above forms a cellular contraction $\UP(f) : \UP(\mb D) \to \mb X_{\pi^:}$ of $L$, which finishes the proof. 
\end{proof}

We can consider any two reduced trivalent plabic graphs to be equivalent if they can be related by only white and black trivalent moves. Then there is a \emph{square flip graph}, whose vertices are equivalence classes of plabic graphs for each connectivity $\pi^:$, and whose edges connect equivalence classes of graphs which have a pair of representative elements related by a square move. By Theorem~\ref{connected}, the square flip graph is connected for every decorated permutation $\pi^:$. Our result can be restricted to the square flip graph as follows.

\begin{corollary}\label{sqmovecycles}
Let $\mb Y_{\pi^:}$ be the 2-complex given by the square flip graph for plabic graphs with connectivity $\pi^:$, with the following 2-cells glued to it:
\begin{itemize}
\item A quadrilateral, wherever there is a 4-cycle generated by two square moves occurring in separate parts of a plabic graph;
\item A pentagon, wherever there is a 5-cycle generated by five square moves which take place in a subgraph which forms a plabic graph with connectivity $\pi{(5,2)}$ or $\pi{(5,3)}$.
\end{itemize}
Then $\mb Y_{\pi^:}$ is simply connnected.
\end{corollary}

\begin{proof}
Any loop $\gamma$ in $\mb Y_{\pi^:}$ can be extended to a loop $\gamma'$ in $X_{\pi^:}$ by adding the necessary extra white and black trivalent moves. Contract $\gamma'$ to a point step-by-step by moving it across the 2-cells in $\mb X_{\pi^:}$. Each 2-cell in $\mb X_{\pi^:}$ corresponds to either a point or a 2-cell in $\mb Y_{\pi^:}$, so $\gamma$ may also be continuously deformed while maintaining the correspondence between $\gamma$ and $\gamma'$. Then once $\gamma'$ has been deformed to a point, so has $\gamma$. Therefore $\mb Y_{\pi^:}$ is simply connected.
\end{proof}

\section{Triple Crossing Diagrams}\label{tcd}

Dylan Thurston~\cite{ThuD17} introduced \emph{triple crossing diagrams} as a generalization of the domino tilings and their flip operation. Just as the space of domino tilings is flip-connected~\cite{Thu90}, so is the space of (minimal) triple crossing diagrams with a given connectivity~\cite{ThuD17}. We will consider only what Thurston~\cite{ThuD17} calls \emph{minimal} triple crossing diagrams, defined as when introduced by Postnikov to study perfect orientations of plabic graphs~\cite{Pos06}.

\begin{definition}\label{tcdef}
Consider a disk with boundary vertices labeled $b_1,b_1',\ldots,b_n,b_n'$ in clockwise order. A \emph{triple crossing diagram} with connectivity (strand permutation) $\pi \in S_n$ consists of $n$ oriented strands drawn inside the disk which start at $b_i$ and end at $b_{\pi(i)}'$ for each $i \in [n]$, satisfying the following properties

\begin{enumerate}
\item Wherever two strands intersect, exactly three distinct strands meet in a \emph{triple crossing}.
\item When considered in cyclic order, the orientation of the six rays from any triple crossing alternates.
\item The diagram contains no \emph{bad double crossings}, defined as when two distinct strands both arrive at triple crossing $c_1$ followed by triple crossing $c_2$.
\end{enumerate}
\end{definition}

It follows from~\cite[Theorem 7]{ThuD17} that this definition is equivalent to Thurston's definition for minimal triple crossing diagrams.

Similar to plabic graphs, triple crossing diagrams on $n$ strands have a connectivity $\pi\in S_n$ given by the final positions of the strands. There is also a notion of flip in a triple crossing diagram, the $\tcd$ move, shown in Figure~\ref{tcdflipng}. Dylan Thurston~\cite[Theorem 5]{ThuD17} proved that all minimal triple crossing diagrams with the same connectivity can be related by a series of $\tcd$ moves.

\begin{figure}[ht]\centering
\includegraphics[width=6in]{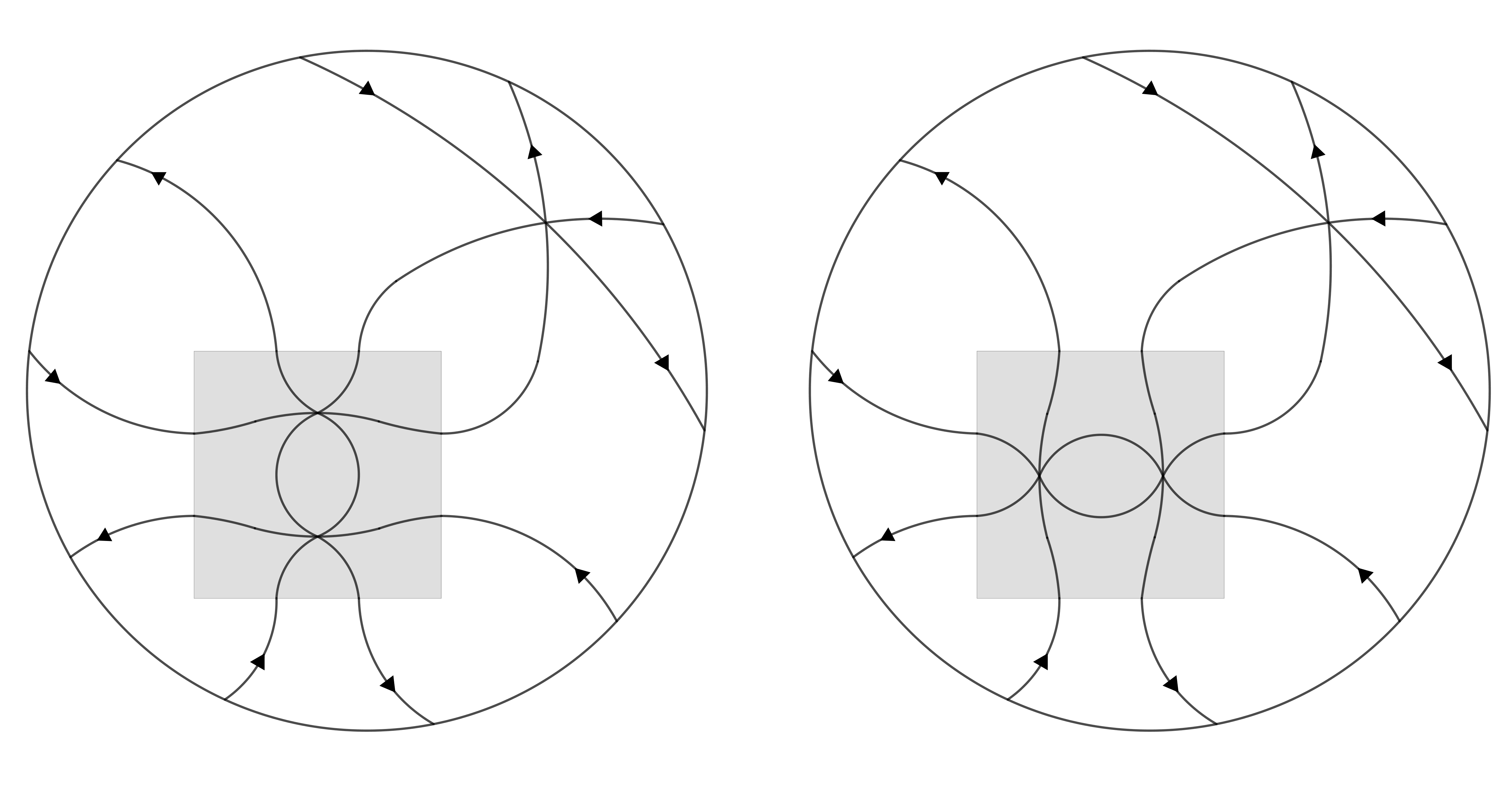}
\caption{A $\tcd$ move (shaded) in a triple crossing diagram}
\label{tcdflipng}
\end{figure}

Postnikov gave the following correspondence between triple crossing diagrams and plabic graphs. For any triple crossing diagram $D$, the plabic graph $\phi(D)$ has white vertices corresponding to triple crossings in $D$, black vertices corresponding to regions bounded by counterclockwise-oriented strands that aren't in the middle of a possible $\tcd$ move, and edges corresponding to counterclockwise regions bordered by triple crossings and to triple crossing which could be involved in a $\tcd$ move together.

\begin{lemma}[{\cite[Lemma 14.4]{Pos06}}]
The map $\phi$ described above gives a bijection between triple crossing diagrams with strand connectivity $\pi$ and reduced plabic graphs for the connectivity $\pi$ (fixed points undecorated) with all white vertices trivalent and no edges with both endpoints black.
\end{lemma}

Such plabic graphs can be considered to be trivalent plabic graphs where the configuration of the edges between black vertices is arbitrary (choose any sequence of uncontraction moves on the black vertices with degree more than three). We observe that the flips in the two contexts correspond nicely

\begin{lemma}\label{tcdflips}
Let $D$ and $D'$ be triple crossing diagrams related by a single $\tcd$ move in $D$. Then $\phi(D)$ and $\phi(D')$ are related by a square move and several black contraction/uncontraction moves if the interior region of the $\tcd$ move was oriented clockwise, otherwise they are related by a single white trivalent move. Conversely, if $G$ and $G'$ are reduced plabic graphs with all white vertices trivalent, and no edges with both endpoints black which are related by a single white trivalent move or a square move and several black contraction/uncontraction moves, then $\phi^{-1}(G)$ and $\phi^{-1}(G')$ are related by a single $\tcd$ move.
\end{lemma}\begin{proof}
Examine how $\phi$ transforms $\tcd$ moves and $\phi^{-1}$ transforms plabic moves locally.
\end{proof}

Dylan Thurston~\cite{ThuD17} conjectured the following, which we now prove as a theorem

\begin{theorem}\label{tcdconj}
Let $\mb T_\pi$ be the 2-complex given by the flip graph of triple crossing diagrams with connectivity $\pi$, with the following $2$-cells glued to it (cf. \cite[Figure 4]{ThuD17}) \begin{itemize}
\item A quadrilateral, wherever two flips are commuting in different parts of the diagram;
\item A pentagon, wherever there is a 5-cycle taking place in a subset of the diagram which is a triple crossing diagram with connectivity $\pi{(5,1)}$ or $\pi{(5,3)}$;
\item A decagon, wherever there is a 10-cycle taking place in a subset of the diagram which is a triple crossing diagram with connectivity $\pi{(5,2)}$.
\end{itemize}
Then $\mb T_\pi$ is simply connected for all permutations $\pi$.
\end{theorem}
\begin{proof}
Let $\gamma$ be a cycle $D_1,D_2,\ldots,D_{m+1} = D_1$ of triple diagrams with connectivity $\pi$ related by $\tcd$ moves. Then let $\phi(\gamma)$ be the cycle $G_1,G_2,\ldots,G_{M+1} = G_1$ of plabic graphs in $\mb X_{\pi}$ constructed by Lemma~\ref{tcdflips} which contains $\phi(D_1),\phi(D_2),\ldots,\phi(D_{m+1}) = \phi(D_1)$ in that order. By Theorem~\ref{plabicycles}, $\phi(\gamma)$ can be continuously deformed to a point in $\mb X_\pi$ by moving it across the cells in $X_\pi$. By construction of $\mb T_\pi$ and Lemma~\ref{tcdflips}, the cells in $X_\pi$ correspond to either points or cells in $\mb T_\pi$. In particular, if $\phi(\gamma)'$ is a deformation of $\phi(\gamma)$ from moving across a cell in $X_\pi$, then there is a (possibly trivial) cell in $\mb T_\pi$ which $\gamma$ can be moved across to create $\gamma'$ such that $\phi(\gamma') = \phi(\gamma)'$. Then after deforming $\phi(\gamma)$ to a point in $\mb X_\pi$ while doing the corresponding deformations to $\gamma$, the cycle $\gamma$ must also have been deformed to a point. Therefore $\mb T_\pi$ is simply connected.
\end{proof}

\section{Generalized Baues Problem}\label{baues}

Recall the model example of a flip graph, whose vertices are the triangulations of a convex polygon, and whose edges are the flips between triangulations, defined by flipping a diagonal in any quadrilateral. The relations between flips are understood well in this case (see Appendix below). One can also consider the relations between relations, and so on, and all those higher relations lead to the classical construction of Stasheff's associahedron, which captures in a sense the ``topology of triangulations''. A different way to study this topology comes from considering the map $p: C(n,d) \to C(n,2)$ projecting the $d$-dimensional cyclic polytope onto the convex $n$-gon. The triangulations of $C(n,2)$ arise as the homeomorphic images of certain two-dimensional subcomplexes of $C(n,d)$, and the associahedron can be recovered in this interpretation as the Minkowski average of the fibers of the map $p$, as introduced in~\cite{BilStu92}. We can go further and consider coarser subdivisions of $C(n,2)$ as well, in a sense induced by the map $p$, then form a poset out of them, and ask questions about its topology. The \emph{generalized Baues problem} is such a question. It was posed in~\cite{BilKapStu94}, for arbitrary linear maps of polytopes, as a generalization of original Baues's question from~\cite{Baues80}, where the target was one-dimensional. Even though in general the GBP-conjecture is false~\cite{RamZie96}, there are positive results in some special cases. The GBP for $p: C(n,d) \to C(n,2)$ was solved affirmatively in~\cite{RamSan00}; the solution for the GBP for the zonotopal tilings, arising from the projection of a hypercube, follows from the results of~\cite{StuZie93}.

In this section, we interpret Theorem~\ref{plabicycles} in terms of the GBP for the cyclic projection of a hypersimplex. We pose the relevant special case of the GBP, following~\cite{Pos18}. The hypersimplex $\Delta_{kn}$ is the $k$-th cross-section of the $n$-dimensional hypercube:
$$
\Delta_{kn} = \left\{(x_1, \ldots, x_n) \in \mathbb{R}^n ~\middle\vert~ 0 \le x_i \le 1, \sum x_i = k\right\}.
$$

If we affinely project the $n$-hypercube onto $\mathbb{R}^3$ so that the image is the cyclic zonotope $Z(n,3)$ (this is achieved by sending the $i$-th basis vector to $v_i \in \mathbb{R}^3$, where the $v_i$ are as in Definition~\ref{cyczonotope}) then the hypersimplex $\Delta_{kn}$ goes to the $k$-th horizontal layer of $Z(n,3)$, which we call $Q_{kn}$:
$$
Q_{kn} = Z(n,3) \cap \{(k,x_2,x_3) \in \mathbb{R}^3\}.
$$
The affine projection $\pi: \Delta_{kn} \to Q_{kn}$ gives rise to a family of \emph{$\pi$-induced subdivisions} of $Q_{kn}$. By definition, a subdivision $\Sigma$ of $Q_{kn}$ into convex polygons is $\pi$-induced if there is a polyhedral subcomplex $\mathfrak{S}$ of $\Delta_{kn}$ such that $\pi$ establishes an isomorphism between $\mathfrak{S}$ and $\Sigma$, as polyhedral complexes.
The $\pi$-induced subdivisions form a poset, called the \emph{Baues poset} and denoted by $\omega(k,n,2)$ in~\cite{Pos18}, whose order relation comes from the inclusion relation on the subcomplexes $\mathfrak{S}$ of $\Delta_{kn}$. The minimal elements of $\omega(k,n,2)$ correspond to the finest $\pi$-induced subdivisions, arising from the subcomplexes of the $2$-skeleton of $\Delta_{kn}$ that get projected down onto $Q_{kn}$ homeomorhically. We readily recognize those subdivisions as the tilings dual to the trivalent plabic graphs of connectivity $\pi(n,k)$~\cite{Gal17}. The maximal element $\hat{1}$ of $\omega(k,n,2)$ is unique and corresponds to the trivial subdivision of $Q_{kn}$, consisting of a single 2-face. The GBP in this case asks the following.

\begin{question}[Special case of~{\cite[Problem~10.3]{Pos18}}]\label{gbp}
Does the poset $\omega(k,n,2) - \hat 1$ (the Baues poset with the maximal element removed) have the homotopy type of the $(n-4)$-dimensional sphere? \footnote{While this paper was under review, this question was answered in the affirmative~\cite[Theorem~1.1]{olarte2019} by Olarte and Santos.}
\end{question}

Another part of~{\cite[Problem~10.3]{Pos18}} asks whether the poset $\omega(k,n,2)$ coincides with another poset of interest, arising from the zonotopal tilings of $Z(n,3)$. Every zonotopal tiling $\Delta$ (not necessarily a fine one) induces a tiling of $\Delta_{kn}$ in its $k$-th horizontal section. Those tilings of $\Delta_{kn}$, ordered by refinement, form the poset of \emph{lifting} subdivisions of $\Delta_{kn}$, denoted as $\omega_{\lift}(k,n,2)$. After Theorem 11.7 in~\cite{Pos18}, Postnikov asks if $\omega_{\lift}(k,n,2)$ coincides with $\omega(k,n,2)$. After an auxiliary lemma, we prove that this is indeed so.

\begin{lemma}\label{corner} 
Let $\Sigma$ be a plabic tiling of the domain inside a necklace $\mc I$ of cyclic connectivity $\phi(\mc I) = \pi(d,k)$. If $k\neq 1$ then no white region of $\Sigma$ has three vertices in common with $\mc I$. If $k\neq d-1$ then no black region of $\Sigma$ has three vertices in common with $\mc I$.
\end{lemma}
\begin{proof}
Recall that $\phi(\mc I) = \pi(d,k)$ means that $\mc I$ consists of the cyclic (modulo $d$) shifts of $\{1,\ldots, k\} \subset [d]$. If $S\cup\{a\}, S\cup\{b\}, S\cup\{c\}$ are three sets in $\mc I$ bounding a white region, one readily sees that $k=1$. The second half of the lemma follows similarly.
\end{proof}

\begin{lemma}\label{lift}
$\omega(k,n,2) = \omega_{\lift}(k,n,2)$.
\end{lemma}
\begin{proof}
Fix a Grassmannian graph $G$ with helicity $k$ on $n$ boundary vertices. It suffices to show that it is realized as the $k$-th cross-section of some (not fine) tiling of $Z(n,3)$. Galashin~\cite{Gal17} proved this for $G$ plabic, and Lemma~\ref{fixing} in the Appendix implies it when $G$ is almost plabic. Say $v_i$ are the vertices of $G$, each with some helicity $k_i$ and degree $d_i$. Choose arbitrary plabic graphs $G_i$ with connectivity $\pi(d_i,k_i)$, for each $i$. Then we can form a plabic graph $G'$ by replacing each vertex $v_i$ in $G$ with $G_i$. 

Each subgraph $G_i$ is the $k_i$-th cross-section of some fine tiling of $Z(d_i,3)$. We will construct  a fine zonotopal tiling $\Delta$ which contains all of these fine tilings $Z(d_i,3)$, by describing its cross-sections layer-by-layer. In particular, it will follow that the $Z(d_i,3)$ do not overlap. Start with the $k$-th layer, which ought to contain the plabic graph $G'$. Suppose for induction that we have described the $j$-th layer (and it agrees with corresponding layers of the $Z(d_i,3)$), but we have not described the $(j+1)$-st layer. By Lemma~\ref{compatible}, the $(j+1)$-st layer is determined by the $j$-th layer, up to the triangulation of the white regions. Then all of the vertices of the $Z(d_i,3)$ tilings will appear as needed in layer $j$ no matter what, but some of the edges may depend on the triangulation. 

Let $\Sigma_i$ be the intersection of $Z(d_i,3)$ with the $(j+1)$-st layer. We only look at the non-trivial $\Sigma_i$ (non-empty, not a single point). Every such $\Sigma_i$ is a convex $d_i$-gon. We want the boundary of $\Sigma_i$ to be respected by the triangulation that we choose. The black triangles in layer $(j+1)$ are in one-to-one correspondence with the white triangles in layer $j$. It follows that the overlap $\Sigma_i \cap \Sigma_{i'}$, for  $i \neq i',$ is white, if non-empty (otherwise $Z(d_i,3)$ and $Z(d_{i'},3)$ have an overlap in the $j$-th layer). Then Lemma~\ref{corner} implies that $\Sigma_i$ and $\Sigma_{i'}$ cannot overlap, for otherwise one of their boundaries would have too many common vertices with their white overlap region. Then we can triangulate the white regions to match the boundary of $\Sigma_i$ (and the chosen tiling of $Z(d_i,3)$) independently for all $i$. The white regions outside of $\bigcup \Sigma_i$ can be triangulated in an arbitrary way.





We conclude that we can choose cross-sections for the layers $k$ and above which agree with the cross-sections of the chosen fine tilings $Z(d_i,3)$. A symmetric argument shows that we can do the same for the layers $k$ and below, completing the description of $\Delta$, which has the crucial property that $\Delta$ respects the boundaries of the $Z(d_i,3)$.

Now, the sub-tilings of these copies of $Z(d_i,3)$ can all be replaced with individual copies of a single tile, $Z(d_i,3)$. The new tiling $\Delta'$ has the Grassmannian graph $G$ as its $k$-th cross-section, since the effect of replacing the tilings of $Z(d_i,3)$ with single tiles in the $k$-th cross-section is to contract the graphs $G_i$ to the single vertices $v_i$.
\end{proof}

Having Lemma~\ref{lift} at our disposal, we interpret Theorem~\ref{plabicycles} in terms of the GBP to give some evidence in favor of the affirmative answer to Question~\ref{gbp}.

\begin{theorem}\label{gbpevidence}
If $n \ge 6$, the poset $\omega(k,n,2) - \hat 1$ is simply connected.
\end{theorem}
\begin{proof}
By Lemma~\ref{lift}, we can work with the zonotopal sections instead of Grassmannian graphs.

The poset of Grassmannian graphs is not graded, but we still can introduce the rank function as the length of the longest chain of covering relations finishing at a plabic trivalent graph. Then the rank 0 graphs are just plabic trivalent graphs; the rank 1 graphs are \emph{almost plabic} in the notation of~\cite{Pos18} and correspond to plabic moves; the rank 2 graphs correspond to the 2-cells of the complex $\mb X_{\pi^:}$ in Theorem~\ref{plabicycles}. 
    
    
Consider the sub-poset $\omega_{\le 2}$ of $\omega(k,n,2) - \hat 1$ consisting of those three types of graphs of rank at most 2. The condition $n \ge 6$ implies that the excluded singleton-graph $\hat 1$ was of rank higher than 2. 
    
The nerve (or the order complex) of $\omega_{\le 2}$ is isomorphic to the barycentric subdivision\footnote{The barycentric subdivision of a regular CW-complex $X$ is the order complex for the closure poset of $X$, that is, for the poset of the closed cells of $X$, ordered by inclusion.} of $\mb X_{\pi^:}$, as simplicial complexes. This is done by a straightforward identification of the low rank graphs with the cells of $\mb X_{\pi^:}$, as above. By Theorem~\ref{plabicycles}, $\omega_{\le 2}$ is simply connected. 

Consider the 2-skeleton $\omega^{(2)}$ of the nerve of $\omega(k,n,2) - \hat 1$. It contains $\omega_{\le 2}$. We show that $\omega^{(2)}$ is simply connected. Let $L$ be a (simplicial) loop in $\omega^{(2)}$. We would like to pull it down (in the sense of rank) along triangles of $\omega^{(2)}$, to a loop in $\omega_{\le 2}$, and then contract it there. We start by breaking $L$ into intervals starting and finishing at local minima (in the sense of rank). For each local minimum $G$, pull it down to a rank 0 graph $\widehat G$, in an arbitrary fashion. For each interval $G_1 - \ldots - G_m$ between two consecutive minima $G_1$ and $G_m$, find its maximum $G_i$. Connect the graphs $\widehat G_1$ and $\widehat G_m$ inside the poset $\{\text{graphs} \underset{\text{Baues}}\le G_i, \text{ of rank 0 or 1}\}$ by a chain $\widehat G_1 = H_1 - \ldots - H_\ell = \widehat G_m$ (it can be done by~\cite[Proposition~11.4]{Pos18}). Note that the chain $G_1 - \ldots - G_m$ can be pulled down to the chain $\widehat G_1 - H_1 - \ldots - H_\ell - \widehat G_m$ along the triangles $G_1 G_2 \widehat G_1, \ldots, G_{i-1} G_i \widehat G_1,  H_1 G_i H_2, \ldots, H_{\ell-1} G_i H_\ell, \widehat G_m G_i G_{i+1}, \ldots, \widehat G_m G_{m-1} G_m$. Repeating this over all intervals between the local minima, we homotope $L$ to a loop inside $\omega^{(2)}$. But $\omega^{(2)}$ is simply connected, so $\omega(k,n,2) - \hat 1$ is simply connected as well.
\end{proof}


\section{Open Questions}\label{questions}

\begin{enumerate}[wide, labelindent=0pt, label=\textbf{(\arabic*)}]

\item A \emph{(single) wiring diagram} is a way to write the completely inverted permutation $w_0$ as a product of $\binom{n}{2}$ elementary transpositions $s_i$ in $S_n$, and the flip (or mutation) operation is the \emph{Coxeter move}, $s_is_{i+1}s_i \leftrightarrow s_{i+1}s_is_{i+1}$, which can connect any two wiring diagrams. There is one-to-one correspondence between wiring diagrams and the rhombus tilings of the regular $2n$-gon, so the proof of Henriques and Speyer in~\cite{HenSpe10} applies to wiring diagrams as well and shows that simple cycles of wiring diagram mutations are of length 4 and 8.
    
    
We can ask the same question for \emph{double wiring diagrams}, introduced in \cite{Fomin00}, which are formed by two interlaced single wiring diagrams. The moves shown in Figure~\ref{dwdflip} can relate any two double wiring diagrams, and are akin to the square move in plabic graphs. 

Any double wiring diagram on $n$ strands can be realized as a plabic graph with strand connectivity $\pi(2n,n)$, but the converse is false for $n > 3$. Corollary~\ref{squareflips} would suggest that there is a simply-connected complex of double wiring diagrams with $2$-cells glued along $4$- and $5$-cycles. However, since not all plabic graphs are double wiring diagrams, we can also see $8$-cycles like those in single wiring diagrams, as can be seen in the flip graph for $n=4$. We conjecture that filling in these $8$-cycles is enough to make the complex simply connected.

\begin{figure}[ht]\centering
\includegraphics[width=0.75\textwidth]{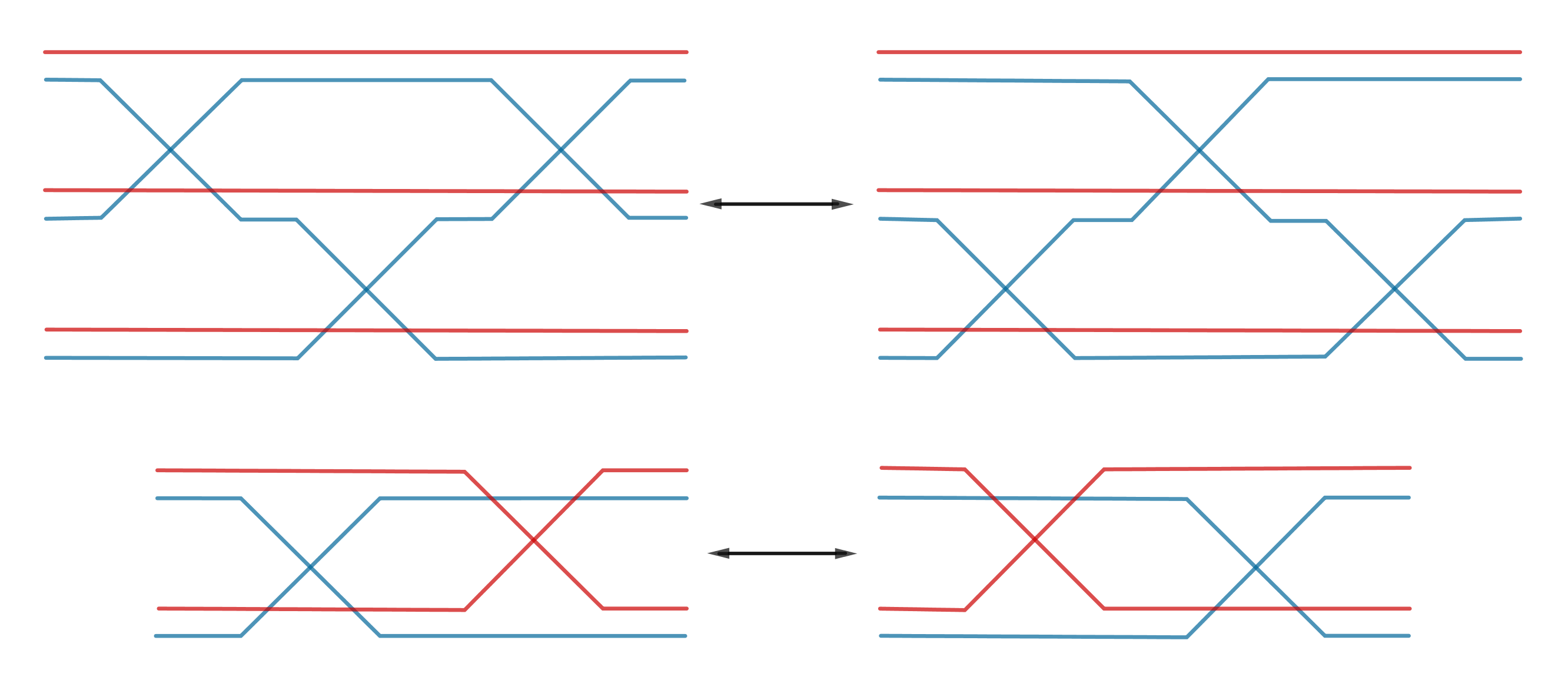}
\caption{Double wiring diagram flips, with red Coxeter move omitted}
\label{dwdflip}
\end{figure} 

\item In the appendix below we reprove a major case of Theorem~\ref{plabicycles} by extending plabic graphs to zonotopal tilings, and relating the cycles of plabic moves to the cycles of zonotopal flips. Can one prove the general case of Theorem~\ref{plabicycles} in a similar way, by relating general plabic graphs to the zonotopal tilings of some tileable subdomains of $Z(n,3)$? The uniqueness of maximal/minimal tilings, as in Theorem~\ref{bruhat}, no longer holds for general tileable domains. 

\item
When defining $\mb X_{\pi^:}$ in Section~\ref{cycles}, we glued certain 2-cells, corresponding to the relations between plabic moves, to the flip graph of plabic configurations. It seems that the complexes $\mb X_{\pi{(6,i)}}$, $1\le i \le 5$, are just 2-spheres. One can proceed and glue 3-cells to $\mb X_{\pi^:}$ along the following 2-spheres:
\begin{itemize}
    \item every copy of $\mb X_{\pi{(6,i)}}$, $1\le i \le 5$, occuring inside $\mb X_{\pi^:}$;
    \item every copy of $\partial \left( \mb X_{\pi{(5,i)}} \times \mb X_{\pi{(4,j)}} \right)$, $1\le i \le 4$, $1\le j \le 3$;
    \item every copy of $\partial \left( \mb X_{\pi{(4,i)}} \times \mb X_{\pi{(4,j)}} \times \mb X_{\pi{(4,\ell)}} \right)$, $1\le i,j,\ell \le 3$.
\end{itemize}
If the 3-complex gotten from $\mb X_{\pi{(7,i)}}$, $1\le i \le 6$, happens to be a 3-sphere (which we expect), we can proceed by gluing 4-cells in a similar fashion, and so on. 
\begin{question}
Is the $n$-complex, obtained from $\mb X_{\pi{(n+4,i)}}$, $1\le i \le n+3$, by gluing cells up to dimension $n$ as above, homeomorphic to the $n$-sphere?
\end{question}

This question seems to be closely related to the generalized Baues problem~\ref{gbp}. 

\item Is there any nice combinatorial interpretation for the cross-sections of higher zonotopes $Z(n,\ge 4)$? An analogue of plabic triangulation in the 3-dimensional sections of $Z(n,4)$ is a certain type of tesselations by tetrahedra and octahedra, while the counterpart of the square move switches between two specific subdivisions of a certain 9-vertex polytope (however, this changes the number of tiles). Methods in this paper can be used to describe the cycles in the flip graph of such configurations.

\item The flip graph of zonotopal tilings of a general (not cyclic) zonotope might be disconnected, even in three dimensions~\cite{liu2018}. What can be said about its flip cycles? Are they still ``generated'' by cycles of length 4 and 10?
\end{enumerate}

\subsection*{Acknowledgements}

This work began as a project in the MIT Summer Program for Undergraduate Research (SPUR) in summer 2018, by the second-named author under the mentorship of the first-named author. 
We thank Pavel Galashin and Alexander Postnikov for suggesting this problem to us, and for fruitful discussions and interest in our work. We thank Jorge Olarte and Francisco Santos, who pointed out a gap in the earlier version of the proof of Lemma~\ref{lift}. We also thank an anonymous referee for valuable remarks.

\appendix
\setcounter{secnumdepth}{0}
\renewcommand{\thesection}{A}
\section{Appendix}\label{sec:appendix}
\setcounter{theorem}{0}

The proof of Theorem~\ref{plabicycles} in Section~\ref{cycles} relies on the base case, which can be reformulated as follows. Consider the CW-complex $\mb U_{n}$, defined as follows
\begin{itemize}
    \item The vertices of $\mb U_n$ are the triangulations of a convex $n$-gon.
    \item The edges of $\mb U_n$ are the flips between triangulations differing inside a quadrilateral.
    \item Every five triangulations differing inside a pentagon give rise to a 5-cycle of flips, along which we glue a 2-cell in $\mb U_n$.
    \item Every two commuting triangulations occurring in non-overlapping quadrilaterals give rise to a 4-cycle of flips, along which we glue a 2-cell in $\mb U_n$.
\end{itemize}
\begin{fact}\label{triang}
The complex $\mb U_{n}$ is simply connected.
\end{fact}

This is equivalent to the case $k=1$ of Theorem~\ref{plabicycles}. Fact~\ref{triang} is well-known, and goes back at least to Stasheff's work~\cite{Sta63}, where he constructs his celebrated associahedron. The formulation as above (but in greater generality) could be found, for example, in~\cite[Theorem~7]{LubMasWag17}. An independent proof of this fact follows from the argument below. 

We give a different proof of Theorem~\ref{plabicycles} for the case of cyclic connectivity $\pi^: = \pi(n,k)$. An advantage of this proof that it doesn't rely on the base case $k=1$, unlike the one in Section~\ref{cycles}. 

Our strategy is to consider a cycle of plabic moves as a cycle of zonotopal tiling flips, and apply Theorem~\ref{zonotopecycles}. We would perform each plabic move by doing flips in tiling containing the graph as a cross-section. Unfortunately the appropriate flip isn't always available, but luckily we can set it up without changing the relevant layer. Let $\Delta$ be a zonotopal tiling and $G_k$ be a plabic graph formed by a cross-section of $\Delta$. 

\begin{lemma}\label{cleanmove} Suppose $M$ is a possible black (resp. white) trivalent move in $G_k$. Then there exists a finite sequence of flips $(S_1,S_2,\ldots,S_m)$ in $\Delta$, such that $G_\ell$ is unchanged by each of the first $m-1$ mutations for any $\ell$ at least (resp. at most) $k$, but the move $M$ occurs on the last mutation.
\end{lemma}\begin{proof}
Complementing all of the labels of the vertices doesn't change the structure of the available flips but does change the colors of all of the regions, so it suffices to prove the result when $M$ is a black trivalent move. We proceed by induction on $k$. When $k \leq 2$, there are no legal black trivalent moves, so the claim holds vacuously. Now, the black trivalent move corresponds to two black triangles in $\Sigma_k$, which by Lemma \ref{compatible} creates two white triangles in $\Sigma_{k-1}$, which are forced to border two black regions. If the black regions are triangulated such that a square move is legal using the white triangles, then perform the corresponding flip and we're done. Otherwise, there exists a sequence of triangulation flips in the black regions which would make the square move legal. By the inductive hypothesis, each of these flips can be done through a finite sequence of mutations, each of which (except the last) leave $G_\ell$ unchanged for all $\ell \geq k-1$. The last flip in each sequence performs a black trivalent move in $\Sigma_{k-1}$, so also leaves $S_k$ unchanged. Therefore we can set up the square move in $\Sigma_{k-1}$ without changing $\Sigma_k$ at all, so the induction is complete.
\end{proof}

In order to properly embed cycles as cyclic zonotopal flips, we also need to match up tilings which share a cross-section. Once we've done that, we are ready to prove the result for cyclic connectivities. \begin{lemma}\label{fixing}
Let $\Delta$ and $\Delta'$ be two fine zonotopal tilings of $Z(n,3)$ which are identical on $G_k$ for some fixed $k$. Then there exists a series of flips, none of which alter $G_k$, which transform $\Delta$ into $\Delta'$.
\end{lemma}\begin{proof}
It suffices to find such a sequence of moves which make $\Delta$ match $\Delta'$ on $\Sigma_{k+1}$ (without ever changing $G_k$ or any lower layer) and $\Sigma_{k-1}$ (without ever changing $G_k$ or any higher layer). Once this is done we can recursively match all of the layers to transform $\Delta$ into $\Delta'$. By Lemma~\ref{compatible}, $\Delta$ and $\Delta'$ already agree up to white (resp. black) triangulation on $\Sigma_{k+1}$ (resp. $\Sigma_{k-1})$. By the flip connectivity of triangulations, there exists a sequence of white (resp. black) trivalent flips in $G_{k+1}$ (resp. $G_{k-1}$) which transform $\Delta$ to completely match $\Delta'$ on $\Sigma_{k+1}$ (resp. $\Sigma_{k-1}$). By Lemma~\ref{cleanmove}, for each of these flips there exists a finite sequence of flips which perform only this move in $G_{k+1}$ (resp. $G_{k-1}$, none of which change $G_\ell$ for any $\ell \leq k$ (resp. $\ell \geq k$). Therefore all of these triangulation moves can be performed without ever changing $G_k$ or any lower (resp. higher) layer, as desired.
\end{proof}

\begin{proof}[Proof of Theorem~\ref{plabicycles} for cyclic permutations.]

Fix any cyclic permutation $\pi^: = \pi(n,k)$. Let $\gamma = M_1M_2\cdots M_{m}$ be a loop in $\mb X_\pi^:$ connecting plabic graphs $G_k^1,G_k^2,\ldots,G_k^{m+1} = G_k^1$ with connectivity $\pi(n,k)$. We will construct a loop $Z(\gamma)$ in $\mb Z_{n,3}$ such that the flips in $Z(\gamma)$ cause exactly the moves $M_1,M_2,\ldots,M_m$ to occur in $G_k$, in that order. For $i = 0,1,\ldots,m-1$, there exists $\Delta_i$ whose cross-section at height $k$ is exactly $G_k^{i+1}$. By Lemma~\ref{squareflips} (if $M_{i+1}$ is a square move) and Lemma~\ref{cleanmove} (if $M_{i+1}$ is a black or white trivalent move), there exists a sequence of moves starting from $\Delta_i$ which performs only the move $M_{i+1}$ in $G_k^{i+1}$. The resulting tiling $\Delta_i'$ from this sequence of moves is identical to $\Delta_{i+1}$ at height $k$, so by Lemma~\ref{fixing} there exists another sequence flips, none of which cause a move in $G_k$, which turns $\Delta_i'$ into $\Delta_{i+1}$, where $i+1$ is considered modulo $m$. Concatenating all these sequences of moves results in our loop $Z(\gamma)$ with the desired properties.

By Theorem~\ref{zonotopecycles}, the loop $Z(\gamma)$ is contractible, by moving it across the 2-cells in $\mb Z_{n,3}$. We will show that these 2-cells correspond to 2-cells in $\mb X_{\pi(n,k)}$ nicely, in order to contract $\gamma$.

The quadrilaterals in $\mb Z_{n,3}$ are formed by two commuting flips in $Z(n,k)$, which result in either two moves in separate parts of $G_k$ (a quadrilateral in $\mb X_{\pi(n,k)})$, one move being performed twice in $G_k$ (an edge in $\mb X_{\pi(n,k)}$), or no moves in $G_k$ (a point in $\mb X_{\pi(n,k)}$). In all cases, when $Z(\gamma)$ is moved across a quadrilateral, the image of the quadrilateral in $\mb X_{\pi(n,k)}$ is a vertex, edge, or 2-cell which $\gamma$ can also be moved across.

The only other 2-cells in $\mb Z_{n,3}$ are decagons whose vertices correspond to the ten refinements of an instance of $Z(5,3)$ inside $Z(n,3)$. Depending on where the plane $x=k$ intersects the copy of $Z(5,3)$, one of five things could happen in $G_k$ as the ten flips in the decagon are performed (see Figure~\ref{plabicyclepng}). \begin{enumerate}
\item If $x=k$ does not intersect the copy of $Z(5,3)$ or only touches the top or bottom vertex, no moves occur in $G_k$ and the image of the decagon in $\mb X_{\pi(n,k)}$ is a vertex.
\item If $x=k$ intersects the copy of $Z(5,3)$ at relative height $1$, then five white trivalent moves occur in a subgraph of $G_k$ with connectivity $\pi{(5,1)}$. The image of the decagon in $\mb X_{\pi(n,k)}$ is a pentagon.
\item If $x=k$ intersects the copy of $Z(5,3)$ at relative height $2$, then five square moves and five white trivalent moves occur in a subgraph of $G_k$ with connectivity $\pi{(5,2)}$. The image of the decagon in $\mb X_{\pi(n,k)}$ is another decagon.
\item If $x=k$ intersects the copy of $Z(5,3)$ at relative height $3$, then five square moves and five black trivalent moves occur in a subgraph of $G_k$ with connectivity $\pi{(5,3)}$. The image of the decagon in $\mb X_{\pi(n,k)}$ is another decagon.
\item If $x=k$ intersects the copy of $Z(5,3)$ at relative height $4$, then five black trivalent moves occur in a subgraph of $G_k$ with connectivity $\pi{(5,4)}$. The image of the decagon in $\mb X_{\pi(n,k)}$ is a pentagon.
\end{enumerate}

In all cases, when $Z(\gamma)$ is moved across the decagon, the image of the decagon is a vertex or 2-cell in $\mb X_{\pi(n,k)}$ which $\gamma$ can be moved across. 

Finally, let $Z(\gamma)'$ be a deformation of $Z(\gamma)$ by moving it across a $2$-cell. We have considered all possible 2-cells in $\mb Z_{n,3}$ and shown that there always exists a cell in $\mb X_{\pi(n,k)}$ which $\gamma$ can be moved across to create $\gamma'$ such that $Z(\gamma') = Z(\gamma)'$. Therefore by contracting $Z(\gamma)$ to a point in $\mb Z_{n,3}$ step-by-step while adjusting $\gamma$ along the way, $\gamma$ is also contracted to a point. 

We note that that the procedure of moving across the $2$-cells in fact builds a cellular contraction, cf. Definition~\ref{cellcontract}.
\end{proof}
This method of proof cannot be straightforwardly applied to the more general statement of Theorem~\ref{plabicycles}. Although a loop of flips for any connectivity $\pi^:$ could still be included in a loop of zonotopal tiling flips, the contraction of the loop might not stay inside the graph with connectivity $\pi^:$ (see the second question in Section~\ref{questions}). 

\newpage

\bibliographystyle{plain}
\bibliography{references}

\end{document}